\documentclass[3p]{elsarticle}

\usepackage{hyperref, amsmath, amsfonts, amsthm}

\journal{arXiv}

\newtheorem{theorem}{Theorem}[section]
\newtheorem{lemma}[theorem]{Lemma}
\newtheorem{prop}[theorem]{Proposition}
\newtheorem{fact}[theorem]{Fact}

\newtheorem{cor}[theorem]{Corollary}
\newdefinition{defn}[theorem]{Definition}

\begin{document}

\begin{frontmatter}
\title{Accessible Proof of Standard Monomial Basis for Coordinatization\\ of Schubert Sets of Flags \footnote{To be contained in the author's doctor thesis written under the supervision of Robert A. Proctor.}}
\author{David C. Lax}
\address{Department of Mathematics, The University of North Carolina at Chapel Hill\\Phillips Hall CB\#3250, Chapel Hill, NC 27599, USA}
\ead{dclax@live.unc.edu}
\begin{abstract}
The main results of this paper are accessible with only basic linear algebra.
Given an increasing sequence of dimensions, a flag in a vector space is an increasing sequence of subspaces with those dimensions.
The set of all such flags (the flag manifold) can be projectively coordinatized using products of minors of a matrix.
These products are indexed by tableaux on a Young diagram.
A basis of ``standard monomials" for the vector space generated by such projective coordinates over the entire flag manifold has long been known.  
A Schubert variety is a subset of flags specified by a permutation.
Lakshmibai, Musili, and Seshadri gave a standard monomial basis for the smaller vector space generated by the projective coordinates restricted to a Schubert variety.
Reiner and Shimozono made this theory more explicit by giving a straightening algorithm for the products of the minors in terms of the right key of a Young tableau.
Since then, Willis introduced scanning tableaux as a more direct way to obtain right keys.
This paper uses scanning tableaux to give more-direct proofs of the spanning and the linear independence of the standard monomials.
In the appendix it is noted that this basis is a weight basis for the dual of a Demazure module for a Borel subgroup of $GL_n$.
This paper contains a complete proof that the characters of these modules (the key polynomials) can be expressed as the sums of the weights for the tableaux used to index the standard monomial bases.
\end{abstract}
\end{frontmatter}

\section{Introduction}

The main results of this paper are accessible to anyone who knows basic linear algebra:  the Laplace expansion of a determinant is the most advanced linear algebra technique used.
Otherwise, the most sophisticated fact needed is that the application of a multivariate polynomial may be moved inside a limit.
Readers may replace our field $\mathbb{C}$ with any field of characteristic zero, such as $\mathbb{R}$.

Let $n \geq 2$ and $1 \leq k \leq n-1$.  
Fix $0< q_1 < q_2< \dots < q_k < n$ and let $Q$ denote the set $\{q_1, \dots, q_k\}$.
A $Q$-flag of $\mathbb{C}^n$ is a sequence of subspaces $V_1 \subset V_2 \subset \dots \subset V_k \subset \mathbb{C}^n$ such that dim($V_j$) $= q_j$ for $1\leq j \leq k$.
The set $\mathcal{F}\ell_Q$ of $Q$-flags has long been studied by geometers.
It is known as a \emph{flag manifold (for $GL_n$)}.
Given a fixed sequence of integers $\zeta_1 \geq \zeta_2 \geq \dots \geq \zeta_m$ with $\zeta_i \in Q$ for $1 \leq i \leq m$, one  can form projective coordinates for $\mathcal{F}\ell_Q$ as follows:
First, any flag can be represented with a sequence of $n$ column vectors of length $n$.
The juxtaposition of these vectors forms an $n \times n$ matrix $f$.
For each $1\leq i \leq m$, form a left-initial $\zeta_i \times \zeta_i$ minor of $f$ by selecting $\zeta_i$ of its $n$ rows.
We refer to a product of such minors as a ``monomial" for the given $\zeta_j$'s.
Let $N$ be the number of such possible monomials.
One can inefficiently coordinatize $\mathcal{F}\ell_Q$ in $\mathbb{P}(\mathbb{C}^N)$ by evaluating all of these monomials over the flag manifold.
The sequence $\zeta_1, \dots, \zeta_m$ can be viewed as the lengths of the columns of a Young diagram $\lambda$.
By the 1950s it was known how to use the semistandard Young tableaux on the diagram $\lambda$ to index an efficient subset of these coordinates;   this has been attributed to Young and to Hodge and Pedoe.
This subset is a basis of ``standard" monomials for the vector space generated by all monomials over the flag manifold.
One can group flags into subsets known as \emph{Schubert varieties} using a form of Gaussian elimination on their matrix representatives; these can be indexed by $n$-permutations.
For a given Schubert variety, the coordinatization by the set of monomials indexed by semistandard tableaux is inefficient.
Utilizing recent developments in tableau combinatorics, this paper gives a new derivation of a basis of standard monomials for the vector space generated by all monomials restricted to a Schubert variety. 

The most famous flag manifolds are the sets of $d$-dimensional subspaces of $\mathbb{C}^n$.
These are the cases $k:=1$ and $q_1:=d$ above and are known as the Grassmannians.
Here the basis result for Schubert varieties may be readily deduced once it is known for the entire Grassmannian.
The next-most studied flag manifold is the ``complete" flag manifold, which is the case $k:=n-1$ above.

It was not until the late 1970s that Lakshmibai, Musili, and Seshadri first gave \cite{LMS} a standard monomial basis for any  Schubert variety of a general flag manifold (for $GL_n$).
Their solution used sophisticated geometric methods and was expressed in the language of the representation theory of semisimple Lie groups.
In 1990, Lascoux and Sch\"utzenberger defined \cite{LS} the ``right key" of a semistandard tableau.
In 1997, Reiner and Shimozono used the notion of right key to give \cite{RS2} a new derivation of the standard monomial basis for any Schubert variety of the complete flag manifold.
They provided a ``straightening algorithm" for products of minors that expressed the monomial specified by a given tableau as a linear combination in the standard monomial basis.
In 2013, Willis defined \cite{Willis} the ``scanning tableau"  of a semistandard tableau and showed that it is the right key of Lascoux and Sch\"utzenberger.
The scanning tableau appears to be the simplest description of the right key.

We show how scanning tableaux can be used to improve the proofs of \cite{RS2} for the spanning and the linear independence of the standard monomials.
All aspects of our presentation  consider all Schubert varieties of all flag manifolds for $GL_n$, i.e. for any $1 \leq k \leq n-1$.
The statements of our basis theorem, Theorem~\ref{main}, and both its spanning and linear independence parts differ from the analogous statements in \cite{RS2}:
We do not limit ourselves to the $k=n-1$ complete flag case.
Here we use the scanning tableau to determine whether the monomial of a given tableau is a member of our standard basis for a given Schubert variety.
In that article, membership is determined by using a ``jeu de taquin" procedure to compute the right key of a tableau.
The use of scanning tableaux allows for a direct and widely accessible proof of this theorem which is entirely self-contained.
As a consequence of our basis theorem, we obtain a weighted tableaux summation expression, Corollary~\ref{character}, that is associated to the vector space at hand.
It is the ``Demazure polynomial" of \cite{PW}, or the ``key polynomial" of \cite{RS1} (which is given in terms of right keys).
The derivation of this character expression is also self-contained:
In particular, the original notion of right key is not needed.

Our spanning proof uses scanning tableaux to give a straightening algorithm in the spirit of \cite{RS2}.
The determinantal identity from \cite{Turnbull} used there is also used here; more details are given for its application to the projective coordinates of a Schubert variety.
Combinatorialists' interest in straightening algorithms goes back at least to \cite{DRS, DKR}.
Apart from motivation, the spanning proof does not need any mention of $Q$-flags or Schubert varieties.
All of the necessary definitions for the spanning theorem, Theorem~\ref{span 2}, make sense for matrices with entries from any commutative ring $R$.
The theorem statement itself makes sense over $R$ when ``spans" is replaced by ``generates as an $R$-module."
The proof presented in this paper is valid at that  level of generality.

Our linear independence proof follows the general inductive strategy used in \cite{LMS} and \cite{RS2}.
However, the simpler combinatorics of scanning tableaux allow those proofs to be simplified.
One simpler aspect is that now only single Schubert varieties need be considered in the induction, rather than the unions of Schubert varieties that arose in the earlier papers.
The statement of the linear independence theorem, Theorem~\ref{LI}, makes sense over any field.
The proof presented here is valid for any field of characteristic zero; we make this assumption to obtain a self-contained development.
The related proof in \cite{RS2} does not need characteristic zero since it refers to a standard fact concerning the closure of a ``Bruhat cell."
There it is assumed the base field is algebraically closed, but given \cite{Humph}, they actually do not need that assumption for this fact.
Hence the basis results in \cite{RS2} and here hold over any field.
See the appendix for details.

We need a number of well-known facts about Schubert varieties for our linear independence proof.
There are references for these facts at varying levels of sophistication for the Grassmannians \cite{Procesi, LB} or the complete flag manifold \cite{MS}.
However, we have not found a comprehensive source at any level of sophistication.
Nor have we found a combination of sources that are accessible to readers without advanced educations in pure mathematics.
So we have included elementary proofs of these standard facts for all flag manifolds for $GL_n$:
Sections \ref{SETUP}, \ref{DEFS}, \ref{PREF}, and \ref{MON} of this paper can serve as an accessible introduction to the subject.
The appendix provides an interface with the modern literature on flag manifolds and Schubert varieties.
Using this appendix, the reader can transition from this  paper to the reductive Lie group and representation theory contexts of references such as \cite{Procesi, LB}.
There we describe how the standard monomial basis provides a basis of global sections for a certain line bundle on a homogeneous space of $GL_n$. 
This is a weight basis for the dual of a Demazure module for a Borel subgroup of $GL_n$.  
For coordinatizing Schubert varieties, it is sufficient to consider Young diagrams with columns of length less than $n$.
Such diagrams would also suffice if one were interested only in realizing representations of $SL_n$.
But we allow our Young diagrams to have columns of length $n$ so that we can realize all of the irreducible polynomial representations of $GL_n$ in the appendix.

Combinatorial tools are introduced in Section~\ref{SETUP}.  
Section~\ref{DEFS} presents the definitions of flag varieties, Schubert varieties, and their projective coordinates.  
Our main theorem, Theorem~\ref{main}, is motivated and stated there.
Sections~\ref{SPAN1} and \ref{SPAN2} prove the spanning parts of Theorems \ref{tableaux} and \ref{main}.
Sections~\ref{PREF} and \ref{MON} present the facts needed to projectively coordinatize Schubert varieties.  
Section~\ref{LIND} proves the linear independence parts of Theorems \ref{tableaux} and \ref{main}.
Section~\ref{CHAR} presents the Demazure polynomial summation.
Section~\ref{APP} is the appendix of contemporary terminology.

\section{Combinatorial tools}\label{SETUP}

The needed combinatorial tools are ``$Q$-chains", which we use to index Schubert varieties, and ``tabloids", which we use to index some projective coordinates for flag manifolds.

Fix $n \geq 2$ and a nonempty subset $Q\subseteq\{1, 2, \dots, n-1\}$ throughout the paper.
Set $k:=|Q|$ and index the elements of $Q$ in increasing order: $1\leq q_1 < q_2 < \dots <q_k < n =: q_{k+1}$.
Define $[n]:=\{1,2, \dots, n\}$.
A \emph{$Q$-chain} is a sequence of subsets $ P_1 \subset P_2 \subset \cdots \subset  P_k \subseteq [n]$ such that $|P_j|=q_j$ for $1 \leq j \leq k$.

An \emph{$n$-partition} is an $n$-tuple $\lambda=(\lambda_1, \lambda_2, \dots, \lambda_n)$ satisfying $\lambda_1 \geq \lambda_2 \geq \dots \geq \lambda_n \geq 0:=\lambda_{n+1}$.
Fix an $n$-partition $\lambda$.
The \emph{shape} of $\lambda$, also denoted $\lambda$, is an array of $n$ rows of boxes that has $\lambda_r$ boxes in row $r$.
The column lengths of the shape $\lambda$ are denoted $n \geq \zeta_1 \geq \dots \geq \zeta_{\lambda_1}$.
Denote the set of distinct column lengths of $\lambda$ that are less than $n$ by $Q(\lambda)$.
Refer to a location in $\lambda$ with column index $1 \leq c \leq \lambda_1$ and row index $1 \leq r \leq \zeta_c$ by $(r,c)$.
Sets of locations in $\lambda$ are called  \emph{regions}.
A \emph{tabloid} $T$ of shape $\lambda$ is a filling of the shape $\lambda$ with values from $[n]$ such that the values strictly increase down each column.
The value of $T$ at location $(r,c)$ is denoted $T(r,c)$.
Partially order the tabloids of shape $\lambda$ by defining $T \preceq U$ if $T(r,c) \leq U(r,c)$ for all locations $(r,c)\in \lambda$.
We use the term \emph{column tabloid} to refer to a tabloid of shape $1^d$ for some \emph{length} $d \leq n$.
Given a subset $P \subseteq [n]$, define $Y(P)$ to be the column tabloid of length $|P|$ filled with the values of $P$ in increasing order.
There is a unique column tabloid of length $n$, namely $Y([n])$.
A \emph{(semistandard Young) tableau} is a tabloid whose values weakly increase across each row.
In Theorem~\ref{tableaux} we use tableaux to index the standard monomial basis for a flag manifold.

Given a $Q$-chain $\pi=(P_1, \dots, P_k)$, its \emph{key} $Y(\pi)$ is the tabloid whose shape has one column each of the lengths $q_k, q_{k-1}, \dots, q_1$ and which is obtained by juxtaposing the columns $Y(P_k), Y(P_{k-1}), \dots, Y(P_1)$.
It can be seen that $Y(\pi)$ is a tableau.
The \emph{Bruhat order} on $Q$-chains is the following partial order:
For two $Q$-chains $\rho$ and $\pi$, define $\rho \preceq \pi$ if $Y(\rho)\preceq Y(\pi)$.
The \emph{$Q$-carrels} for an $n$-tuple are the following $k+1$ sets of positions:  the first $q_1$ positions, the next $q_2-q_1$ positions, and so on through the last $n-q_k$ positions.
To each $Q$-chain $\pi$, we associate the permutation $\overline \pi$ of $[n]$: 
In $n$-tuple form, the $Q$-carrels of $\overline \pi$ respectively display the elements of the $k+1$ sets $P_1, P_2 \setminus P_1, \dots, P_k \setminus P_{k-1}, [n]\setminus P_k$, with the elements of each set listed in increasing order.
A \emph{$Q$-permutation} is a permutation of $[n]$ in $n$-tuple form such that the values within each $Q$-carrel increase from left to right.
It is easy to see that the creation of $\overline \pi$ describes a bijection from the set of $Q$-chains to the set of $Q$-permutations.

For $1 \leq i < j \leq n$, define the \emph{reflection} $\sigma_{i j}$ to be the following operator on $Q$-chains:
Let $\pi=(P_1,\dots,P_k)$ be a $Q$-chain.
For $1 \leq \ell \leq k$, form the following sets:
If $i \in P_\ell$ and $j \not \in P_\ell$, set $P_\ell':= \left(P_\ell \setminus \{i\}\right) \cup \{j\}$.
If $j \in P_\ell$ and $i \not \in P_\ell$, set $P_\ell':=\left(P_\ell \setminus \{j\}\right) \cup \{i\}$.
Otherwise, set $ P_\ell ':=P_\ell$.
It can be seen that $P_1' \subset \cdots \subset P_k'$; this is the $Q$-chain $\sigma_{ij}\pi$.
If there exists $1 \leq \ell \leq k$ such that $j \in P_\ell$ and $i \not \in P_\ell$, then $Y(\sigma_{ij}\pi)$ is produced from $Y(\pi)$ by decreasing some values from $j$ to $i$ (and sorting the resulting columns), so $\sigma_{ij}\pi \prec \pi$.

The following lemma says that we can find a reflection to step down in the Bruhat order between two $Q$-chains:
\begin{lemma} \label{stepdown}
Let $\rho, \pi$ be $Q$-chains.  
If $\rho \prec \pi$, then there exists $1 \leq i< j \leq n$ such that $\rho \preceq \sigma_{ij}\pi \prec \pi$.
\end{lemma}
\begin{proof}
Write $\rho=(R_1,\dots , R_k)$ and $\pi=(P_1 ,\dots,P_k)$.
Find the rightmost column where the keys $Y(\rho)$ and $Y(\pi)$ differ:  these columns are $Y(R_h)$ and $Y(P_h)$ respectively for some $1 \leq h \leq k$.
Find the minimal $i \in R_h \setminus P_h$ and the minimal $j \in P_h \setminus R_h$.
Since $Y(R_h) \prec Y(P_h)$, we have $i<j$.
Form the $Q$-chain $\sigma_{ij} \pi =(P_1', \dots ,P_k')$.
By the above remark $\sigma_{ij}\pi \prec \pi$.

We verify that $\rho \preceq \sigma_{ij}\pi$: 
For the values of $1\leq \ell \leq k$ such that $P_\ell' = P_\ell $, we have $Y(R_\ell) \preceq Y(P_\ell)=Y( P_\ell')$.
For the other values of $\ell$, we have $P_\ell' = (P_\ell \setminus \{j\}) \cup \{i\}$.  
Let $1 \leq p \leq q_\ell $ denote the row index of the value $j$ in $Y(P_\ell)$.
In the rows below row $p$ the value in $Y(P_\ell')$ is the same as the value in $Y( P_\ell)$ since these are the values of $ P_\ell'$ and $P_\ell$ which are greater than $j$.
So here the value in $Y(R_\ell)$ is at most the value in $Y( P_\ell')$.
For rows at and above row $p$, the value in $Y(R_\ell)$ is at most the value in $Y(P_\ell')$ since $R_\ell$ contains all of the $p$ values of $P_\ell'$ which are less than $j$.
\end{proof}

Fix an $n$-partition $\lambda$ with $Q(\lambda) \subseteq Q$.
For $1\leq \ell \leq k$, the number of columns of length $q_\ell$ in $\lambda$ is $\lambda_{q_\ell}-\lambda_{q_{\ell+1}}$.
The number of columns of length $n$ in $\lambda$ is $\lambda_n$.
Given a $Q$-chain $\pi=(P_1, \dots, P_k)$, its \emph{$\lambda$-key} is the tableau $Y_\lambda(\pi)$ of shape $\lambda$ obtained by juxtaposing $\lambda_n$ copies of the column $Y([n])$, $\lambda_{q_k}-\lambda_{q_{k+1}}$ copies of the column $Y(P_k)$, $\lambda_{q_{k-1}}-\lambda_{q_k}$ copies of the column $Y(P_{k-1})$, $\dots$ and $\lambda_{q_1}-\lambda_{q_2}$ copies of the column $Y(P_1)$.

\begin{lemma}\label{order}
Let $\rho, \pi$ be $Q$-chains.
If $\rho \preceq \pi$ , then $Y_\lambda(\rho) \preceq Y_\lambda(\pi)$.
When $Q(\lambda)=Q$ the converse holds:  if $Y_\lambda(\rho) \preceq Y_\lambda(\pi)$, then $\rho \preceq \pi$.
\end{lemma}
\begin{proof}
Every column of length $n$ is $Y([n])$, and every column of $Y_\lambda(\pi)$ of length less than $n$ appears in $Y(\pi)$.
When $Q(\lambda)=Q$, every column of $Y(\pi)$ also appears in $Y_\lambda(\pi)$.
\end{proof}

We now describe the scanning algorithm of \cite{Willis}.
Fix a sequence $(b_1, b_2, b_3, \dots)$.
Define its \emph{earliest weakly increasing subsequence} (EWIS) to be the subsequence $(b_{i_1}, b_{i_2}, b_{i_3}, \dots)$, where $i_1=1$ and for $j>1$ the index $i_j$ is the smallest index such that $b_{i_j} \geq b_{i_{j-1}}$.
For any tableau $T$ of shape $\lambda$, construct its \emph{scanning tableau} $S(T)$ as follows:
Begin with an empty shape $\lambda$.
Form the sequence of the bottom-most values of the columns of $T$ from left to right.
Find the EWIS of this sequence.
When a value is added to the EWIS, mark its location in $T$.
The sequence of locations just marked is called a \emph{scanning path}.
Fill the lowest available location of the leftmost available column of $S(T)$ with the last member of the EWIS.
Iterate this process as if the marked locations are no longer part of $T$.
Using the row condition on the filling of $T$, it can be seen that at each stage the unmarked locations form the shape of some $n$-partition.
This implies that every location in $T$ is marked once the leftmost column of $S(T)$ has been filled.
To find the values of the next column of $S(T)$:
\begin{enumerate}
\item Ignore the leftmost column of $T$ and $\lambda$.
\item Remove the marks from the remaining locations.
\item Repeat the above process.
\end{enumerate}
Continue until the shape has been completely filled with values: this is the scanning tableau $S(T)$ of $T$.
For a location $(r,c)\in \lambda$, let $P(T;r,c)$ denote the scanning path found to fill location $(r,c)$ of $S(T)$.

We need four lemmas concerning the scanning tableau $S(T)$ of a tableau $T$ of shape $\lambda$.
Only the first is needed to prove Theorem~\ref{span 2}, the main spanning theorem.
The other three along with Lemmas~\ref{stepdown} and \ref{order} are used in Sections~\ref{MON} and \ref{LIND}.
\begin{lemma} \label{scan path}
 Let $1\leq c \leq \lambda_1$ and $1\leq r \leq \zeta_c - 1$.
For any location $(p , b)$ in the scanning path $P(T;r,c)$, there exists a location $(u, v)$ in the previous scanning path $P(T;r+1,c)$ such that $v \leq b$ and $T(u,v)>T(p,b)$.
\end{lemma}
\begin{proof}
Since the scanning algorithm is defined recursively for column bottoms, we may reduce to the case that $(r+1,c)$ is a column bottom of $T$.

 First, suppose $(p, b)$ is a column bottom of $T$.
The location $(p, b)$ is not in  $ P(T;r+1,c)$, but it does belong to the sequence of column bottoms of $T$ which is scanned to form $P(T;r+1,c)$.
Hence its value $T(p,b)$ was not in that previous earliest weakly increasing subsequence.
Therefore there is a column bottom $(u, v)$ of $T$ in the scanning path $P(T;r+1,c)$ strictly to the left of $(p,b)$ such that $T(u, v)>T(p,b)$.

  Now suppose $(p,b)$ is not a column bottom of $T$.  
Since $(p,b)$ is scanned in the formation of $P(T;r,c)$, the location $(p+1, b)$ was marked as part of the previous scanning path $P(T;r+1,c)$.
By the column strict condition on tabloids, its value satisfies $T(p+1, b)>T(p,b)$.
Take $u:=p+1$ and $v:=b$.
\end{proof}

\begin{lemma}\label{key}
Every value in the rightmost column of $T$ appears in every column of $S(T)$.
In particular, the rightmost column of $S(T)$ is the rightmost column of $T$.
\end{lemma}

\begin{proof}
Fix a column index $1\leq c \leq \lambda_1$.
As was noted above, every location in $T$ to the right of column $c$ is marked in the construction of column $c$ of $S(T)$.
So every location in the rightmost column of $T$ belongs to a scanning path $P(T;r,c)$ for some $1\leq r \leq \zeta_c$.
These locations must be the end of their respective scanning paths.
\end{proof}

Let $\lambda'$ denote the partition obtained from $\lambda$ by omitting the rightmost column of its shape.
Given a tableau $T$ of shape $\lambda$, let $T'$ denote the tabloid of shape $\lambda'$ obtained by omitting the rightmost column of $T$.

\begin{lemma}\label{delete}
Deleting the rightmost column both before and after forming the scanning tableau, we find that $S(T')\preceq [S(T)]'$.
\end{lemma}

\begin{proof}
Let $(r,c)\in \lambda'$.
In the two applications of the scanning algorithm, the same locations are marked and removed from within the region $\lambda'\subset \lambda$ of $T$ as are from $T'$.
So the scanning path $P(T;r,c)$ is the path $P(T';r,c)$ with at most one location appended from the rightmost column of $T$.
Since the values within a scanning path weakly increase, the value at the end of $P(T';r,c)$ is less than or equal to the value at the end of $P(T;r,c)$.
The value at location $(r,c)$ in $S(T')$ is the value at the end of $P(T';r,c)$, and the value at $(r,c)$ in $S(T)$ is the value at the end of $P(T;r,c)$.
\end{proof}

\newpage

Now we fix a $Q$-chain $\pi$.
Form its $\lambda$-key $Y_\lambda(\pi)$.
In this $\pi$-specific environment, the notion of tableau is more complicated:
\begin{defn}
A tableau $T$ of shape $\lambda$ is \emph{$\pi$-Demazure} if its scanning tableau satisfies $S(T)\preceq Y_\lambda(\pi)$.
\end{defn}
In Theorem~\ref{main} we use $\pi$-Demazure tableaux to index the standard monomial basis for the Schubert variety indexed by $\pi$.

\begin{lemma} \label{chopD}
If a tableau $T$ of shape $\lambda$ is $\pi$-Demazure, then the tableau $T'$ of shape $\lambda'$ is $\pi$-Demazure.
\end{lemma}

\begin{proof}
By the previous lemma, we have $S(T')\preceq [S(T)]' \preceq [Y_\lambda(\pi)]' = Y_{\lambda'}(\pi)$.
\end{proof}

\section{Flags of subspaces and tabloid monomials}\label{DEFS}

We now introduce the main objects of the paper:  flags of subspaces, Bruhat cells, Schubert varieties, and tabloid monomials.
For Sections~\ref{SPAN1} and \ref{SPAN2},  only the definitions concerning tabloid monomials are needed.
Along the way we mention five facts about these structures for motivation which are formally stated and proved in Sections~\ref{PREF} and \ref{MON}.
Our main result, Theorem \ref{main}, is stated at the end of this section.

\begin{defn}
A \emph{$Q$-flag} of $\mathbb{C}^n$ is a sequence of subspaces $V_1 \subset V_2 \subset \cdots \subset V_k \subset \mathbb{C}^n$ such that $\dim(V_j)=q_j$ for $1 \leq j \leq k$.
\end{defn}
We denote the set of $Q$-flags in $\mathbb{C}^n$ by $\mathcal{F}\ell_Q$.
An ordered basis $(v_1, v_2, \dots, v_n)$ of column vectors for $\mathbb{C}^n$ is presented in this paper as the $n \times n$ invertible matrix $[v_1, v_2, \dots, v_n]$ whose columns from left to right are $v_1, \dots, v_n$.
Define a map $\Phi_Q$ from the ordered bases for $\mathbb{C}^n$ to $\mathcal{F}\ell_Q$ by sending an ordered basis $f:=[v_1, \dots, v_n]$ to the $Q$-flag $\Phi_Q(f)$ of subspaces $V_j=$span$(\{v_i| i \leq q_j \})$ for $1 \leq j \leq k$.
Any $Q$-flag can be represented in this way by many ordered bases.
Special $Q$-flags can be made using the axis basis vectors $e_1, \dots, e_n$ for $\mathbb{C}^n$:
For each $Q$-chain $\pi= (P_1 , P_2 , \cdots , P_k)$, construct the \emph{$Q$-chain flag} $\varphi(\pi)$ of subspaces $V_j:=$span$(\{e_i| i \in P_j\})$ for $1 \leq j \leq k$.
Given a $Q$-chain $\pi$, form the $Q$-permutation $\overline \pi$ as in Section~\ref{SETUP}.
Define the $n \times n$ matrix $s_\pi$ to be the permutation matrix whose $(\overline \pi _j, j)$ entry is $1$ for $1\leq j \leq n$.
It is clear that $\varphi(\pi)=\Phi_Q(s_\pi)$, when $s_\pi$ is viewed as an ordered basis.

Let $B$ denote the subgroup of upper triangular matrices within $GL_n$, the group of invertible matrices.
\begin{defn}
Let $\pi$ be a $Q$-chain.  The \emph{Bruhat cell} $C(\pi)$ is the set $\{\Phi_Q(bs_\pi)| b\in B\}$ of $Q$-flags which can be produced from the ordered basis $s_\pi$ for the $Q$-chain flag $\varphi(\pi)$ with the action of the upper triangular matrices.
\end{defn}
We will see (Fact \ref{partition}) that every $Q$-flag belongs to a unique Bruhat cell.
The following disjoint unions of Bruhat cells are important subsets of $\mathcal{F}\ell_Q$:
\begin{defn}
Let $\pi$ be a $Q$-chain.  We define the \emph{Schubert variety} $X(\pi)$ to be the union of cells $\bigsqcup\limits_{\rho \preceq \pi} C(\rho)$.
\end{defn}

Our goal is to develop a coordinatization of $\mathcal{F}\ell_Q$.
Recall that projective space $\mathbb{P}(\mathbb{C}^n)$ is the set of lines through the origin in $\mathbb{C}^n$; hence it is the set $\mathcal{F}\ell_{\{1\}}$ of $\{1\}$-flags.
The set  $\mathbb{P}(\mathbb{C}^n)$ does not have global coordinates in the usual (affine) sense.
But it can be coordinatized by \emph{projective coordinates}:
A point $L \in \mathbb{P}(\mathbb{C}^n)$ is indexed by an equivalence class $[(p_1, p_2, \dots, p_n)]$ of $n$-tuples, where $(p_1, p_2, \dots, p_n)\in \mathbb{C}^n$ is a nonzero point on the line  $L$ and two $n$-tuples $(p_1, p_2, \dots, p_n)$ and $(p_1', p_2', \dots, p_n')$ are equivalent if there is a nonzero $\alpha \in \mathbb{C}$ such that $(p_1, p_2, \dots, p_n)=(\alpha p_1', \alpha p_2', \dots, \alpha p_n')$.

From now on, fix an $n$-partition $\lambda$ such that $Q(\lambda)\subseteq Q$.
Now we begin to form projective coordinates for $\mathcal{F}\ell_Q$ from tabloids of shape $\lambda$.
Let $\mathbb{C}[x_{ij}]$ denote the ring of polynomials in the $n^2$ coordinates of a sequence of $n$ vectors from $\mathbb{C}^n$.
Fix $1\leq p \leq n$.
Let $f$ be an $n \times n$ matrix.
For any $1\leq q \leq n$, define the \emph{$q$-initial submatrix of $f$ with rows $r_1, \dots, r_p$} to be the $p \times q$ matrix whose $i^{th}$ row consists of the first $q$ entries of the $r_i^{th}$ row of $f$.
When $p=q$, in $\mathbb{C}[x_{ij}]$ we form for $f$ its \emph{$q$-initial minor with rows $r_1, \dots, r_q$}: this is the determinant of its $q$-initial submatrix with rows $r_1, \dots, r_q$.

\begin{defn}
Let $T$ be a tabloid of shape $\lambda$.  For each column index $1\leq c \leq \lambda_1$, form  in $\mathbb{C}[x_{ij}]$ the $\zeta_c$-initial minor with indices $T(1,c), \dots, T(\zeta_c, c)$.  
The \emph{monomial} of $T$, denoted by the corresponding Greek letter $\tau$, is the product of these minors.
Let $\pi$ be a $Q$-chain.
In particular, the monomial of the $\lambda$-key $Y_\lambda(\pi)$ is denoted $\psi_\lambda(\pi)$.
\end{defn}

Let $F$ be a $Q$-flag.
We will see (Lemma \ref{projective}) that the sequence of the valuations of all tabloid monomials of shape $\lambda$ on the ordered bases for $F$ is projectively well defined:
Varying the choice of basis $f$ such that $\Phi_Q(f)=F$ will scale all these values equally.
We will also see (Fact \ref{injective}) that when $Q(\lambda)=Q$, this sequence of monomials give a faithful projective coordinatization of the set $\mathcal{F}\ell_Q$ of $Q$-flags.

\begin{defn}
Let $\Gamma_\lambda$ denote the vector subspace of polynomials in $\mathbb{C}[x_{ij}]$ that are linear combinations of the tabloid monomials of shape $\lambda$. 
\end{defn}

While it is useful to consider the set of all tabloid monomials, the following long-known result shows that that set is much larger than is needed to span $\Gamma_\lambda$:

\begin{theorem}  \label{tableaux}
Let $\lambda$ be an $n$-partition.  The monomials of the semistandard tableaux of shape $\lambda$ form a basis of the vector space $\Gamma_\lambda$.
\end{theorem}

Such monomials are called \emph{tableau monomials}.
The spanning and linear independence parts of this basis theorem are reproved here as Theorem~\ref{span 1} and Corollary~\ref{TabInd}.
This theorem implies that when $Q(\lambda)=Q$, the sequence of tableau monomials gives an efficient coordinatization of $\mathcal{F}\ell_Q$.

Now we return to having a $Q$-chain $\pi$ fixed, as at the end of Section~\ref{SETUP}.
Again form its $\lambda$-key $Y_\lambda(\pi)$.
Define the subspace $Z_\lambda(\pi)\subseteq \Gamma_\lambda$ to be the span of the monomials of tabloids $T$ such that $T \not \preceq Y_\lambda(\pi)$.

\begin{defn}
Let $\pi$ be a $Q$-chain.
The \emph{Demazure quotient} for $\pi$ is the vector space $\Gamma_\lambda(\pi):=\Gamma_\lambda/ Z_\lambda(\pi)$.
\end{defn}

We will see (Lemma \ref{zero}) that all tabloid monomials in $Z_\lambda(\pi)$ are zero on the Schubert variety $X(\pi)$.
If moreover $Q(\lambda)=Q$ , then $X(\pi)$ is the zero set in $\mathcal{F}\ell_Q$ of $Z_\lambda(\pi)$ (Fact \ref{zero_set}).
Hence we consider $\Gamma_\lambda(\pi)$ to be the span of the restrictions of the tabloid monomials to $X(\pi)$.
We simply write ``monomial" to refer to the residue of a monomial in $\Gamma_\lambda(\pi)$.
Since the set of tableau monomials is now much larger than is needed to span $\Gamma_\lambda(\pi)$, we need an analog of Theorem \ref{tableaux} for the space $\Gamma_\lambda(\pi)$.
Our main result is a new proof of the following theorem that is  based on the scanning tableaux $S(T)$:

\begin{theorem}\label{main}
Fix a nonempty $Q\subseteq [n-1]$.  Let $\lambda$ be an $n$-partition such that $Q(\lambda)\subseteq Q$ and let $\pi$ be a $Q$-chain.  The monomials of the $\pi$-Demazure tableaux of shape $\lambda$ form a basis of the vector space $\Gamma_\lambda(\pi)$.
\end{theorem}

Such monomials are called \emph{$\pi$-Demazure monomials}.
The spanning and linear independence parts of this basis theorem are Theorem~\ref{span 2} and Theorem~\ref{LI}.
This theorem implies that when $Q(\lambda)=Q$, the sequence of $\pi$-Demazure monomials gives an efficient coordinatization of the Schubert variety $X(\pi)$ of $\mathcal{F}\ell_Q$.

\section{Tableau monomials span $\Gamma_\lambda$} \label{SPAN1}

Before we prove the spanning part of our main result, Theorem~\ref{main}, in the next section, we must first prove the spanning of $\Gamma_\lambda$ by tableau monomials in Theorem~\ref{tableaux}.  
We begin by presenting a translation of a classical determinantal identity into the language of tabloid monomials.   
This is a ``master" identity that we use in two ways to prove the two spanning results by establishing relations amongst certain monomials.  
The idea of both proofs is the same: Using a total order on the set of tabloids, we provide straightening algorithms for applying the master identity.  
Each use of the identity progresses in the same direction under this order.  
The control afforded by the total order implies the termination of the algorithm.  
This is a common strategy; it was also used in \cite{RS2}.

Fix an $n$-partition $\lambda$;  the sets $Q$ and $Q(\lambda)$ play no role in this section.
Fix a tabloid $T$ of shape $\lambda$ and a region $\mu\subseteq\lambda$.
The region $\mu$ selects which locations are ``active" in the master identity.
The multiset of values of $T$ within $\mu$ is denoted $T(\mu)$.
For $1\leq j \leq \lambda_1$, let $T_j$ denote column $j$ of $T$ and let $\mu_j$ denote the intersection of $\mu$ with column $j$ of $\lambda$.
Let $\bar\mu$ denote the region of $\lambda$ complementary to $\mu$.

\begin{defn}
A \emph{$\mu$-shuffle of $T$} is a permutation of the values of $T$ that can be obtained by the composition of two permutations as follows:
First permute the values within the region $\mu$ such that the values within a given column are distinct. 
Then sort the values within each column into ascending order to obtain a tabloid.
\end{defn}

Given a $\mu$-shuffle $\sigma$ of $T$, the resulting tabloid is denoted $T_\sigma$ and its monomial is denoted $\tau_\sigma$.
Let $\epsilon(\sigma)$ denote the sign of $\sigma$ as a permutation.
For a tabloid $T$ with repeated values, it is possible that for $\mu$-shuffles $\sigma_1\neq \sigma_2$ of $T$ we have $T_{\sigma_1}=T_{\sigma_2}$.

 We prepare to construct a square compound matrix $M_\mu(T)$ based on $T$ and $\mu$.
Let $g$ be the $n \times n$ matrix $(x_{ij})$ of $n^2$ indeterminants.
First we split each of  the initial square submatrices whose minors in $g$ form the monomial $\tau$ of $T$ into two rectangular parts:
For each $1 \leq j \leq \lambda_1$ form the $\zeta_j \times |\mu_j|$ ``active" matrix $A_j$  by transposing the $\zeta_j$-initial submatrix of $g$ whose rows are specified by the values of $T_j(\mu)$.
Also form the $\zeta_j \times |\bar\mu_j|$ ``inactive" matrix $N_j$ by transposing the $\zeta_j$-initial submatrix of $g$ similarly specified by the values of $T_j(\bar \mu)$.
The total number of columns in $A_j$ and $N_j$ is $|\mu_j|+|\bar\mu_j|=\zeta_j$.
Let $A_j \sqcup N_j$ denote the $\zeta_j \times \zeta_j$ concatenation of the matrices $A_j$ and $N_j$.
Except for the order of its columns, the matrix $A_j \sqcup N_j$ is the transpose of the $\zeta_j$-initial submatrix specified by the column $T_j$.
So its determinant is the monomial $\tau_j$ of $T_j$, up to a sign.
These $\zeta_j \times \zeta_j$ matrices form the main diagonal blocks of the compound matrix $M_\mu(T)$.

Now in addition let $1 \leq i \leq \lambda_1$.
Form the rectangular matrix $A^{<i>}_j$ by transposing the $\zeta_i$-initial submatrix of $g$ whose rows are specified by the values of $T_j(\mu)$.
Then let $A^{<i>}_j \sqcup 0$ denote the $\zeta_i \times \zeta_j$ concatenation of the matrix $A^{<i>}_j$ with a $\zeta_i \times |\bar\mu_j|$ zero matrix.
These $\zeta_i \times \zeta_j$ matrices form the off-diagonal blocks of the compound matrix $M_\mu(T)$.

 Define the matrix $M_\mu(T)$ to be the $(\zeta_1+\cdots + \zeta_{\lambda_1})$-square compound matrix whose $j^{th}$ diagonal block is $A^{<j>}_j \sqcup N_j$ and whose non-diagonal block in the $(i,j)$ block position is $A^{<i>}_j \sqcup 0$:
$$M_\mu(T) :=\begin{bmatrix}
A^{<1>}_1 \sqcup N_1 & A^{<1>}_2 \sqcup 0 & \cdots & A^{<1>}_{\lambda_1} \sqcup 0\\
A^{<2>}_1 \sqcup 0 & A^{<2>}_2 \sqcup N_2 & \cdots & A^{<2>}_{\lambda_1} \sqcup 0\\
\vdots & & &\vdots\\ 
\vdots & & &\vdots \\
A^{<\lambda_1>}_1 \sqcup 0 & A^{<\lambda_1>}_2 \sqcup 0 & \cdots & A^{<\lambda_1>}_{\lambda_1} \sqcup N_{\lambda_1}\end{bmatrix}$$

The following lemma is our master identity.  
It says that $|M_\mu(T)|$ is a polynomial in $\Gamma_\lambda$.
It is the translation hinted at by Reiner and Shimozono of the left side of the determinantal identity (III.11) in \cite{Turnbull} that they use as the left side of equation (5.3) in \cite{RS2} .  
Note that the product of the determinants of the diagonal blocks of $M_\mu(T)$ is the monomial $\tau$, up to a sign.
This sign is the sign ambiguity in the statement of the lemma.
This ambiguity vanishes in our applications.

\begin{lemma} \label{shuffles}
Let $T$ be a tabloid of shape $\lambda$ and let $\mu \subseteq \lambda$. The determinant $|M_\mu(T)|$ is, up to sign, the signed sum of monomials $\sum \epsilon(\sigma)\tau_\sigma,$ where the sum runs over all $\mu$-shuffles $\sigma$ of $T$.
\end{lemma}

\begin{proof}
We calculate $|M_\mu(T)|$ by an iterated Laplace expansion process.
We then show that the nonzero terms in this expansion correspond to the $\mu$-shuffles of $T$.
Begin to calculate $|M_\mu(T)|$ by Laplace expansion on the first $\zeta_1$ rows, which form the first row of blocks.  
This expresses the determinant as the sum of the products of $\zeta_1 \times \zeta_1$ ``primary" minors and $(\zeta_2+\cdots+\zeta_{\lambda_1})$-square ``complementary" minors.

   Most of the products in this sum vanish because one of the two minors has a zero column or because the primary minor has repeated columns.
Fix a summand that does not vanish for either of these reasons.  
Since the complementary minor at hand cannot have a zero column, its primary minor must include all of the columns of $N_1$.
The primary minor's other columns come from various $A^{<1>}_j$ blocks.
Recall that  the columns of $N_1$ are the initial segments of the rows of $g$ indexed by the values of $T_1(\bar\mu)$ and that the columns of $A^{<1>}_j$ are the initial segments of the rows of $g$ indexed by the values of $T_j(\mu)$.
Because the primary minor does not have repeated columns, the values of $T_1(\bar\mu)$ and those of $T(\mu)$ that correspond to the columns of the $A^{<1>}_j$ blocks that contribute to the primary minor are distinct.
Therefore we may define a column tabloid $U_1$ of length $\zeta_1$ that is filled with all the values of $T_1(\bar\mu)$ together with these values from $T(\mu)$.
Except for the order of its columns, the primary minor is the determinant of the $\zeta_1$-initial submatrix specified by the column $U_1$.
So this primary minor is the monomial $\upsilon_1$ of $U_1$, up to a sign.

 Now consider the complementary minor of our fixed summand.
Begin the next iteration by computing this determinant by Laplace expansion on its first $\zeta_2$ rows, which form its first row of blocks.
Fix a summand as above.
Analogously define a column tabloid $U_2$ of length $\zeta_2$  filled with all the values of $T_2(\bar\mu)$ and the values of $T(\mu)$ corresponding to this new primary minor.
Again this primary minor is the monomial $\upsilon_2$ of $U_2$, up to a sign.
Note that the values of $T(\mu)$ used here come from different locations within $\mu$ than those in the first iteration.

 Continue iterating this Laplace expansion process.
In the end, we find that our fixed nonzero term is the product $\upsilon_1\cdots\upsilon_{\lambda_1}$, up to a sign.
This is the monomial $\upsilon$ of the tabloid $U$ of shape $\lambda$ that is formed by the juxtaposition of the column tabloids $U_1, \dots, U_{\lambda_1}$.
The value from each location of $T$ was used exactly once to construct $U$.
Hence, the tabloid $U$ was formed by a permutation $\sigma$ of the values of $T$.
Express $\sigma$ as a composition of the following two permutations:
First permute the values within $\mu$ so that the values used from $T(\mu)$ for $U_i$ appear in the $i^{th}$ column.
As noted earlier, the values within each column are distinct.
Then sort each column to obtain the tabloid $U$.
Hence the permutation $\sigma$ is actually a $\mu$-shuffle of $T$.
Therefore each nonzero term of this iterated Laplace expansion is the monomial $\tau_\sigma$ for a $\mu$-shuffle $\sigma$ of $T$, up to a sign.
The sign is the product of the signs from the Laplace expansion process and the signs from presenting each $\zeta_i \times \zeta_i$ minor as the monomial of a column tabloid.  
If the sign of the diagonal term $|A_1^{<1>}\sqcup N_1| \cdots |A_{\lambda_1}^{<\lambda_1>}\sqcup N_{\lambda_1}|$ of this expansion agrees with the sign of  the monomial $\tau$ of the identity $\mu$-shuffle of $T$, then for each $\mu$-shuffle $\sigma$ of $T$ the sign of the $\tau_\sigma$ term is $\epsilon(\sigma)$.
Otherwise the diagonal term is $-\tau$; then the sign of each $\tau_\sigma$ term is $-\epsilon(\sigma)$.

It is clear that the $\mu$-shuffles of $T$ associated to any two nonzero terms are distinct.
 It is also true that every $\mu$-shuffle of $T$ is associated to one of these terms:
Fix a $\mu$-shuffle $\sigma$ of $T$.
At step $i$ in the iterated Laplace expansion, choose the primary minor to consist of the block $N_i$ together with the columns from the $A^{<i>}_j$ blocks that correspond to the values of $T(\mu)$ that $\sigma$ moves to column $i$ of $T$.
\end{proof}

We will choose regions $\mu$ based on the tabloid $T$ such that we can show $|M_\mu(T)|=0$.
Here the sign ambiguity vanishes, and Lemma~\ref{shuffles} produces relations among the tabloid monomials.
The choice of $\mu$ below yields a well-known ``Pl\"ucker relation".
The presentation of this rederivation prepares the reader for the proof of the new Proposition~\ref{turnb}.

\begin{prop} \label{snake}
 Let $T$ be a tabloid of shape $\lambda$.  Let $1\leq c \leq \lambda_1 -1$ and $1\leq r \leq \zeta_{c+1}$.  Let $\mu \subseteq \lambda$ be the region $\{(i,c)|r\leq i \leq \zeta_c\}\cup\{(j,c+1)| 1\leq j \leq r\}$. 
Then $\sum \epsilon(\sigma)\tau_\sigma=0$, where the sum runs over the $\mu$-shuffles $\sigma$ of $T$.
\end{prop}

\begin{proof}
Construct the matrix $M_\mu(T)$ as above.  
By Lemma~\ref{shuffles}, it is sufficient show that $|M_\mu(T)|=0$.
Since $\mu$ only has two columns, indexed $c$ and $c+1$, the blocks $A^{<i>}_j$ within $M_\mu(T)$ are empty for $j \neq c$ or $c+1$.
Hence the determinant $|M_\mu(T)|$ simplifies to $\tau_1\dots\tau_{c-1} det(*) \tau_{c+2}\dots\tau_{\lambda_1}$, where $*$ is the $(\zeta_c + \zeta_{c+1})$-square  matrix
\small $\begin{bmatrix} 
A^{<c>}_c \sqcup N_c & A^{<c>}_{c+1} \sqcup 0\\
A^{<c+1>}_c \sqcup 0 & A^{<c+1>}_{c+1} \sqcup N_{c+1}
\end{bmatrix}.$ \normalsize

Note that $\zeta_{c}\geq \zeta_{c+1}$.
Subtract the first $\zeta_{c+1}$ rows of the matrix $*$ from its last $\zeta_{c+1}$ rows to get the matrix\\ 
 \small $\begin{bmatrix} 
A^{<c>}_c \sqcup N_c & A^{<c>}_{c+1} \sqcup 0 \\
0 \sqcup -N^{<c+1>}_c & 0 \sqcup N_{c+1}\end{bmatrix}$\normalsize, where $N_c^{<c+1>}$ is the submatrix of $N_c$ formed by its first $\zeta_{c+1}$ rows.
The determinant is unchanged.
The $|\mu|=\zeta_c+1$ columns \small $\begin{bmatrix} A^{<c>}_c & A^{<c>}_{c+1}\\ 0 &0\end{bmatrix}$\normalsize have only $\zeta_c$ nonzero rows.   
\end{proof}

 Let $U$ and $T$ be column tabloids of the same length.  
If the string of values of $U$ read from top to bottom precedes the string of values of $T$ in lexicographic order, then we define $U \leq T$.
Let $U$ and $T$ be two tabloids of shape $\lambda$. 
 If the string of columns of $U$ read left to right precedes the string of columns of $T$ in lexicographic order, then we define $U \leq T$.  
This is a total order; it extends the partial order $\preceq$ of Section~\ref{SETUP}.

 Our goal is to re-express the monomial of any tabloid $T$ which is not a tableau in terms of monomials of tabloids $U<T$.
The following easy lemma is the first step.

\begin{lemma} \label{sort}
Let $T$ be a tabloid of shape $\lambda$.  Let $U$ be the tabloid obtained by sorting the columns of $T$ of a given length in ascending order according to the total order $\leq$.  Then $U\leq T$, and these tabloids have the same monomial.
\end{lemma}

The relations given by Proposition~\ref{snake} are sufficient for the following:

\begin{prop} \label{downstream}
Let $T$ be a tabloid of shape $\lambda$ which is not a tableau.  Then there exist coefficients $a_U=\pm 1$ such that $\tau=\sum a_U \upsilon,$ where the sum is over some tabloids U such that $U<T$.
\end{prop}

\begin{proof}
If $T$ does not already have its columns sorted by the total order, apply Lemma~\ref{sort} to get $\tau=\upsilon$ where $\upsilon$ is the monomial of a tabloid $U<T$.

 Now suppose $T$ has sorted columns.
Since $T$ is not a tableau, there exists a location $(r,c) \in \lambda$ such that $T(r,c)>T(r,c+1)$.  
With this $c$ and $r$, take $\mu$ as in Proposition~\ref{snake}.
Then the relation $\sum \epsilon(\sigma)\tau_\sigma= 0$ holds, where the sum runs over all $\mu$-shuffles of $T$.
We can solve for $\tau=\tau_{id}$ to obtain $\tau=-\sum \epsilon(\sigma)\tau_\sigma$, where the sum now runs over all non-identity $\mu$-shuffles of $T$.

 Consider a non-identity $\mu$-shuffle $\sigma$ of $T$.  
We show $T_\sigma < T$.
By the column filling property of tabloids and our choice of $\mu$, every value in $T_c(\mu)$ is larger than every value in $T_{c+1}(\mu)$.
Since $\sigma \neq id$, it replaces some value in $T_c(\mu)$ with a value from $T_{c+1}(\mu)$ and then sorts column $c$.
Then in the highest location of column $c$ in which $T$ and $T_\sigma$ differ, the smaller value is in $T_\sigma$.
Therefore column $c$ of $T_\sigma$ precedes $T_c$ in the total order, while all columns to its left are unchanged.
Thus  $T_\sigma<T$.

  Since $\mu$ does not have repeated values, each tabloid $T_\sigma$ is distinct.
Therefore when we sum over tabloids instead of $\mu$-shuffles of $T$, the coefficients of the monomials remain $\pm1$.
\end{proof}

We are ready to prove the spanning part of Theorem \ref{tableaux}.
Elements of this proof reappear in the spanning proof for Theorem~\ref{main}. 

\begin{theorem}\label{span 1}
Let $\lambda$ be an $n$-partition.  The monomials of the semistandard tableaux of shape $\lambda$ span the vector space $\Gamma_\lambda$.
\end{theorem}

\begin{proof}
The space $\Gamma_\lambda$ was defined to be the span of all tabloid monomials.
Given a tabloid $U$,  we show that its monomial $\upsilon$ is in the span of the tableau monomials.  
Suppose that $U$ is not already a tableau.

 Apply Proposition~\ref{downstream} to express $\upsilon$ as a linear combination of monomials of tabloids preceding $U$.
If any of these tabloids is not a tableau, apply Proposition~\ref{downstream} to the largest among them according to the total order $\leq$ and iterate this step.
After each iteration, the largest tabloid which is not a tableau that appears precedes that of the previous iteration.
Since there are finitely many tabloids of shape $\lambda$, this process must terminate.
When the process terminates, we have an expression for $\upsilon$ as a linear combination of tableau monomials.
\end{proof}

\section{Demazure monomials span the Demazure quotient} \label{SPAN2}

This section is a continuation of Section~\ref{SPAN1}.
Returning to the context at the end of Section~\ref{DEFS}, again fix an $n$-partition $\lambda$ with $Q(\lambda) \subseteq Q$ and a $Q$-chain $\pi$.
By Theorem~\ref{span 1}, the space $\Gamma_\lambda(\pi):=\Gamma_\lambda/Z_\lambda(\pi)$ is spanned by the residues of tableau monomials.
Here we re-express the (residue) monomial in $\Gamma_\lambda(\pi)$ of any tableau which is not $\pi$-Demazure by choosing an appropriate region $\mu$ for an application of Lemma~\ref{shuffles}.
We write a bar over a polynomial of $\Gamma_\lambda$ to indicate its residue in $\Gamma_\lambda(\pi)$.
Proposition~\ref{turnb}, Proposition~\ref{downstream 2}, and Theorem~\ref{span 2} are respectively analogous to Proposition~\ref{snake}, Proposition~\ref{downstream}, and Theorem~\ref{span 1}.

 If a tableau $T$ is not $\pi$-Demazure, then there exists a location $(r,c) \in \lambda$ such that $S(T)[r,c]>Y_\lambda(\pi)[r,c]$.
Our region $\mu$ will consist of locations that are associated to each of the locations  $(r,c), (r+1, c), \dots, (\zeta_c, c)$ in column $c$: These $\zeta_c-r+1$ other locations will be indexed by $r, r+1, \dots, \zeta_c$.
The last value of the scanning path $P(T;r,c)$ is $S(T)[r,c]$.
Define $(p_r, b_r)$ to be the location of  the first value in the path $P(T;r,c)$ which is larger than $Y_\lambda(\pi)[r,c]$.
If $(r,c)$ is a column bottom, then take $\mu$ to be the region $\{(p_r,b_r)\}$.
Otherwise do the following:
By Lemma \ref{scan path}, there exists at least one location $(u,v)$ in $P(T;r+1,c)$ such that  $u \leq b_r$ and $T(u,v)>T(p_r, b_r)$.  
Define $(p_{r+1},b_{r+1})$ to be the first such location in $P(T;r+1,c)$.
Continue in this fashion until $(p_{\zeta_c}, b_{\zeta_c})$ in $P(T; \zeta_c, c)$ has been defined: we have found locations $(p_r, b_r),(p_{r+1}, b_{r+1}) , \dots ,(p_{\zeta_c}, b_{\zeta_c})$ with $b_{r}\geq b_{r+1} \geq \dots \geq b_{\zeta_c}\geq c$ and $Y_\lambda(\pi;r,c)<T(p_r, b_r)< T(p_{r+1}, b_{r+1})< \dots < T(p_{\zeta_c}, b_{\zeta_c})$.
Take $\mu$ to be the region $\{ (p_r, b_r) , \dots , (p_{\zeta_c}, b_{\zeta_c}) \}$.

\begin{prop} \label{turnb}
Let $\pi$ be a $Q$-chain.
Let $T$ be a tableau of shape $\lambda$ which is not $\pi$-Demazure.
Let $\mu\subseteq \lambda$ be the region just defined.  Then $\sum \epsilon(\sigma)\overline \tau_\sigma~=~0$ in $\Gamma_\lambda(\pi)$, where the sum runs over the $\mu$-shuffles $\sigma$ of $T$.
\end{prop}

\begin{proof}
Construct the matrix $M_\mu(T)$ as in Section~\ref{SPAN1}.  
By Lemma~\ref{shuffles}, it is sufficient show $\overline{|M_\mu(T)|}=0$ in $\Gamma_\lambda(\pi)$.
We refer to the definitions above pertaining to the region $\mu$.
Since the leftmost column of $\mu$ is column $b_{\zeta_c}\geq c$, the determinant $|M_\mu(T)|$ simplifies to $ \tau_1\dots \tau_{c-1}det(*)$ where $*$ is the lower right $(\zeta_c+\cdots + \zeta_{\lambda_1})$-square submatrix: $$ \begin{bmatrix}
A^{<c>}_c \sqcup N_c & A^{<c>}_{c+1} \sqcup 0 &\cdots & A^{<c>}_{\lambda_1} \sqcup 0\\
A^{<c+1>}_c \sqcup 0 &  A^{<c+1>}_{c+1} \sqcup N_{c+1} & \cdots & A^{<c+1>}_{\lambda_1} \sqcup 0\\
\vdots & \vdots & &  \vdots\\
A^{<\lambda_1>}_c \sqcup 0 &  A^{<\lambda_1>}_{c+1} \sqcup 0 & \cdots &A^{<\lambda_1>}_{\lambda_1} \sqcup N_{\lambda_1}\end{bmatrix}.$$
Because $S(T)$ fails to be dominated by $Y_\lambda(\pi)$ in column $c$, the index $c$ is emphasized over the index $b_{\zeta_c}$.

For $1\leq  j \leq i  \leq \lambda_1$, let $N^{<i>}_j$ denote the submatrix of $N_j$ formed by its first $\zeta_i$ rows.
 For each $c+1\leq i \leq \lambda_1$, the first $\zeta_i$ rows of $*$ are contained in its first row of blocks.
Subtract these rows from its $(i+1-c)^{th}$ row of blocks to get the matrix  $$*':=\begin{bmatrix}
A^{<c>}_c \sqcup N_c &  A^{<c>}_{c+1} \sqcup 0 & \cdots & A^{<c>}_{\lambda_1} \sqcup 0 \\
0 \sqcup -N^{<c+1>}_c & 0 \sqcup N_{c+1} & \cdots & 0\sqcup 0 \\
\vdots & \vdots &  & \vdots \\
0 \sqcup -N^{<\lambda_1>}_c & 0 \sqcup 0 & \cdots & 0\sqcup N_{\lambda_1}\end{bmatrix}.$$
We calculate $det(*')=det(*)$ by Laplace expansion on the first $\zeta_c$ rows, which form its first row of blocks.\\
  Most terms in the expansion vanish.
Fix a term such that neither the primary nor complementary minor has a zero column and the primary minor has no repeated columns.
Since the complementary minor cannot have a zero column, the $\zeta_c \times \zeta_c$ primary minor must use all of the columns of the blocks $A^{<c>}_c,\dots,A^{<c>}_{\lambda_1}$:
 These blocks $A^{<c>}_c,\dots,A^{<c>}_{\lambda_1}$ have a total of only $|\mu|=\zeta_c-r+1$ columns.
Since the primary minor does not have repeated columns, the values of $T_c(\bar \mu)$ that correspond to the $r-1$ columns of $N_c$ that contribute to the primary minor are distinct from the values of $T(\mu)$.
 Therefore we may define a column tabloid $U_c$ of size $\zeta_c$ containing all the values of $T(\mu)$ and these values from $T_c(\bar\mu)$.
Except for the order of its columns, the primary minor is the determinant of the $\zeta_c$-initial submatrix of $g$ specified by the column taboid $U_c$.
So this primary minor is the monomial $\upsilon_c$ of $U_c$, up to a sign.
By the choice of $\mu$, the $\zeta_c-r+1$ values of $T(\mu)$ are all larger than $Y_\lambda(\pi;r,c)$.  
Hence at most $r-1$ values in $U_c$ are less than or equal to $Y_\lambda(\pi;r,c)$.
In particular $U_c(r)>Y_\lambda(\pi;r,c)$.

 Now consider the complementary minor of our fixed summand.
We compute this minor by an iterated Laplace expansion analogous to the one in the proof of Lemma~\ref{shuffles}.
Begin the iteration by computing this determinant by Laplace expansion on its first $\zeta_{c+1}$ rows, which form its first row of blocks.
Fix a summand as above.
Since the current complementary minor cannot have a zero column, this primary minor must include all of the columns of $N_{c+1}$.
The primary minor's other columns come from the $-N^{<c+1>}_c$ block.
Define a column tabloid $U_{c+1}$ of length $\zeta_{c+1}$  filled with all the values of $T_{c+1}(\bar\mu)$ and the values of $T_c(\bar \mu)$ that correspond to the columns of $-N^{<c+1>}_c$ that contribute to the primary minor.
Again the primary minor is the monomial $\upsilon_{c+1}$ of $U_{c+1}$, up to a sign.

 Iterate this process. 
We end up with column tabloids $U_c, U_{c+1}, \dots, U_{\lambda_1}$ of respective lengths $\zeta_c, \zeta_{c+1}, \dots, \zeta_{\lambda_1}$.
We find that our fixed nonzero term is the product $\tau_1\cdots \tau_{c-1} \upsilon_c \cdots \upsilon_{\lambda_1}$, up to a sign.
This is the monomial $\upsilon$ of the tabloid $U$ that is formed by the juxtaposition of the column tabloids $T_1, \dots, T_{c-1}, U_c, \dots, U_{\lambda_1}$.
But $U\not\preceq Y_\lambda(\pi)$, by the observation about $U(r,c)=U_c(r)$ above.
Hence its monomial $\upsilon$ belongs to the subspace $Z_\lambda(\pi)$ of $\Gamma_\lambda$.
Therefore $\overline \upsilon=0$ in $\Gamma_\lambda(\pi)$.
So all of the terms in the iterated Laplace expansion that are not zero in $\Gamma_\lambda$ are in $Z_\lambda(\pi)$.
Hence $\overline{|M_\mu(T)|}=0$.
\end{proof}

 We now show that this result can be used to re-express the monomial of a tableau $T$ which is not $\pi$-Demazure in terms of monomials of tabloids $U<T$.

\begin{prop} \label{downstream 2}
Let $\pi$ be a $Q$-chain.
Let $T$ be a tableau of shape $\lambda$ which is not $\pi$-Demazure.  If $\overline \tau \neq 0$ in $\Gamma_\lambda(\pi)$, then there exist coefficients $a_U=\pm1$ such that $\overline \tau= \sum a_U \overline \upsilon,$
 where the sum is over some tabloids $U$ such that $U<T$.
\end{prop}

\begin{proof}
Since $T$ is not  $\pi$-Demazure, there exists a region $\mu$ as for Proposition~\ref{turnb}.
Then the relation $\sum \epsilon(\sigma)\overline\tau_\sigma= 0$ holds in $\Gamma_\lambda(\pi)$, where the sum runs over all $\mu$-shuffles of $T$.

 If the identity permutation is the only $\mu$-shuffle of $T$, then this equation states that $\overline\tau=0$ in $\Gamma_\lambda(\pi)$.
Otherwise, solve for $\overline\tau = \overline \tau_{id}$.
Consider a non-identity $\mu$-shuffle $\sigma$ of $T$.  
We show $T_\sigma < T$.
Let $b$ be the index of the leftmost column of $\lambda$ such that $\sigma$ replaces some of the values of $T_b(\mu)$.
The replacement values must arrive from later columns of $\mu$.
By our choice of $\mu$, each value in $T_b(\mu)$ is strictly larger than all of the values in the later columns of $\mu$. 
Then in the highest location of column $b$ in which $T$ and $T_\sigma$ differ, the smaller value is in $T_\sigma$.
Therefore column $b$ of $T_\sigma$ precedes $T_b$, while all columns to its left are unchanged.
Thus $T_\sigma<T$.
Since $\mu$ does not have repeated values, each tabloid $T_\sigma$ is distinct. 
Therefore when we sum over tabloids instead of $\mu$-shuffles of $T$, the coefficients of the monomials remain $\pm1$.
\end{proof}

Now we are ready to prove the spanning part of Theorem \ref{main}.

\begin{theorem} \label{span 2}
Fix a nonempty $Q \subseteq [n-1]$.  Let  $\lambda$ be an $n$-partition with $Q(\lambda)\subseteq Q$ and let $\pi$ be a $Q$-chain.  The monomials of the $\pi$-Demazure tableaux of shape $\lambda$ span the  vector space $\Gamma_\lambda(\pi)$.
\end{theorem}

\begin{proof}
By Theorem~\ref{span 1}, the space $\Gamma_\lambda(\pi)$ is spanned by the tableau monomials.
Given a tableau $U$, we show its monomial $\overline \upsilon$ is in the span of the $\pi$-Demazure monomials.
Suppose that $U$ is not already $\pi$-Demazure.

 Apply Proposition~\ref{downstream 2} to express $\overline\upsilon$ as a linear combination of monomials for tabloids preceding $U$.
If any of these tabloids is not a tableau, apply Proposition~\ref{downstream} to the largest among them according to the total order $\leq$ and iterate this step.
As in Theorem~\ref{span 1}, this process terminates.
We have now expressed $\overline \upsilon$ as a linear combination of monomials of tableaux preceding $U$.
If any of the tableaux is not $\pi$-Demazure, apply Proposition~\ref{downstream 2} to the largest among them and then repeatedly apply Proposition~\ref{downstream} to the resulting tabloids.
After each iteration of Propositions \ref{downstream 2} and \ref{downstream}, the largest tabloid which is not a $\pi$-Demazure tableau precedes that appearing in the previous iteration.
Since there are finitely many tabloids of shape $\lambda$, this process must terminate.
When the process terminates, we have an expression for $\overline\upsilon$ as a linear combination of $\pi$-Demazure monomials.
\end{proof}

\section{Preferred bases and Bruhat cells}\label{PREF}

The linear independence of tableau monomials for Theorem~\ref{tableaux} is shown directly in \cite{MS} by organizing the leading terms of tableau monomials with respect to an order on the indeterminants $x_{ij}$ for $\mathbb{C}[x_{ij}]$.
No similarly direct proof for the linear independence part of Theorem~\ref{main} is known.
Instead we assume that our field has characteristic zero and, following \cite{LMS,RS2}, we evaluate a linear combination of monomials at some ordered basis to verify that it is nonzero.
We analyze these evaluations based on the membership of the corresponding $Q$-flag in a Bruhat cell  or a Schubert variety.
So here and in the next section we return to the context of Section~\ref{DEFS} and present the standard facts concerning tabloid monomials, Bruhat cells, and Schubert varieties.
A statement in these sections is displayed as a ``Lemma" if it is needed for Section~\ref{LIND} and as a ``Fact" if it is included only for motivation.
The $n$-partition $\lambda$ plays no role in this section.

Recall that the \emph{$Q$-carrels} for an $n$-tuple are the following $k+1$ sets of positions:  the first $q_1$ positions, the next $q_2-q_1$ positions, and so on through the last $n-q_k$ positions.
Given an ordered basis $f=[v_1, v_2, \dots, v_n]$ of $\mathbb{C}^n$ in matrix form, the formation of the $Q$-flag $\Phi_Q(f)=(V_1, \dots, V_k)$ can be viewed using these $Q$-carrels:
The vectors from the first $Q$-carrel of $f$ span $V_1$, and for $2\leq j \leq k$ the vectors from the $j^{th}$ $Q$-carrel of $f$ extend $V_{j-1}$ to the space $V_j$.
The \emph{pivot} of a nonzero column vector is its last nonzero coordinate.

\newpage

\begin{defn}
An ordered basis $f$ is \emph{$Q$-preferred} if:
\begin{enumerate}
\item Within a $Q$-carrel, the pivots descend from left to right.
\item Each vector $v\in f$ has a value of 1 in its pivot coordinate.
\item All of the coordinate values to the right of a pivot are 0.
\end{enumerate}
\end{defn}

The pivots of a $Q$-preferred basis give information concerning its $Q$-flag:

\begin{lemma}\label{chain}
Let $f=[v_1, \dots, v_n]$ be a $Q$-preferred basis with pivot coordinates $\rho_1, \dots, \rho_n$.
For each $1\leq j \leq k$, define $R_j:=\{\rho_1, \dots, \rho_{q_j}\}$.
Then $\rho= (R_1, \dots ,R_k)$ is a $Q$-chain.
The list $(\rho_1, \dots, \rho_n)$ is the $Q$-permutation $\overline \rho$.
The set $R_j$ is the set of possible pivot coordinates for vectors in $V_j:=$span$(v_1, \dots, v_{q_j})$.
\end{lemma}

\begin{proof}
By the third property of $Q$-preferred bases, the $\rho_1, \dots, \rho_n$ are distinct.
The first conclusion follows immediately.
The condition on the values within the $Q$-carrels of a $Q$-permutation follows for $(\rho_1, \dots, \rho_n)$ from the first property of $Q$-preferred bases;  it is clearly $\overline \rho$.
No nonzero linear combination of vectors with distinct pivot coordinates produces a vector with a new pivot coordinate.
\end{proof}

A $Q$-preferred basis $f$ is a distinctive representative for $\Phi_Q(f)$ in the following way:

\begin{lemma}\label{prefer}
Let $f=[v_1, \dots, v_n]$ be a $Q$-preferred basis with pivot coordinates $\rho_1, \dots, \rho_n$.
Fix $1 \leq m \leq n$, and let $1 \leq j \leq k+1$ be minimal such that $m \leq q_j$.
Then $v_m$ is the unique vector in $V_j:=$span$(v_1, \dots, v_{q_j})$ that has a value of $1$ at its pivot coordinate $\rho_m$ and a value of $0$ at coordinates $\rho_1, \rho_2, \dots, \rho_{m-1}$.
\end{lemma}

\begin{proof}
Let $w$ be any such vector.
Set $u:=w - v_m$.
Since the pivots within a $Q$-carrel of $f$ descend from left to right, we have that $\rho_m < \rho_{m+1} < \dots <\rho_{q_j}$.
Since $w$ and $v_m$ both have pivot coordinate $\rho_m$, the vector $u$ has a value of $0$ at the coordinates $\rho_{m+1}, \dots, \rho_{q_j}$.
Since $w$ and $v_m$ both have a value of $1$ at coordinate $\rho_m$, the vector $u$ has a value of $0$ at coordinate $\rho_m$.
Then $u$ is a vector in $V_j$ with a value of $0$ at the coordinates $\rho_1, \rho_2, \dots, \rho_{q_j}$.
By the preceding lemma, this vector cannot have any other pivot coordinate.
Therefore $u=0$, and so $w=v_m$.
\end{proof}

Each $Q$-flag has a unique $Q$-preferred representative:

\begin{lemma} \label {bijection}
The restriction of the map $\Phi_Q$ to the set of $Q$-preferred bases is a bijection to the set $\mathcal{F}\ell_Q$ of $Q$-flags.
Hence Lemma~\ref{chain} associates to each $Q$-flag a unique $Q$-chain.
\end{lemma}

\begin{proof} 
We begin by showing that the restriction of the map $\Phi_Q$ is injective.
Suppose there are at least two $Q$-preferred bases.
Then $Q \neq \{ n \}$, since here the identity matrix depicts the only $Q$-preferred basis.
Let $f_1=[v_1, v_2, \dots, v_n]$ and $f_2=[w_1, w_2, \dots, w_n]$ be distinct $Q$-preferred bases.
Suppose that for some $1\leq j \leq k$, the sets of pivot coordinates of the first $q_j$ vectors of $f_1$ and of $f_2$ are different.
Let $V_j:=$span$(v_1, \dots, v_{q_j})$ and $W_j:=$span$(w_1, \dots, w_{q_j})$.
Let $1\leq m \leq q_j$ be such that $v_m$ has a pivot coordinate different from those of $w_1, \dots, w_{q_j}$.
By Lemma~\ref{chain}, no vector in $W_j$ has the same pivot coordinate as $v_m$.
Hence $v_m \not \in W_j$, and so $V_j \neq W_j$.
Otherwise the pivot coordinates $\rho_1, \dots, \rho_n$ for $f_1$ and $f_2$ are the same.
Let $1 \leq m \leq n$ be such that $v_m \neq w_m$ and let $1\leq j \leq k+1$ be minimal such that $m \leq q_j$.
The vectors $v_m$ and $w_m$ both have a value of $1$ in their pivot coordinate $\rho_m$ and a value of $0$ in the coordinates $\rho_1 ,\dots, \rho_{m-1}$.
By the preceeding lemma applied to $w_m$, we see that $v_m \not \in W_j$.
So again $V_j\neq W_j$.
Hence $\Phi_Q(f_1) \neq \Phi_Q(f_2)$.

We now provide the inverse map:
Fix a $Q$-flag $F$ and choose \emph{any} ordered basis $h$ such that $\Phi_Q(h)=F$.
The following elementary column operations on $h$ preserve its $Q$-flag:
\begin{enumerate}
\item Swap two columns within the same $Q$-carrel of $h$.
\item Multiply a column of $h$ by a nonzero scalar.
\item Add a multiple of a column of $h$ to a column to its right.
\end{enumerate}
Using these column operations, a Gaussian elimination algorithm can be run on $h$ so that its output $f$ is a $Q$-preferred basis with $\Phi_Q(f)=F$.
We verify that the output is independent of the choice of the basis used to represent $F$:
Suppose two $Q$-preferred bases $f_1$ and $f_2$ can be produced from two representatives $h_1$ and $h_2$ for $F$.
Then $\Phi_Q(f_1)=\Phi_Q(h_1)=\Phi_Q(h_2)=\Phi_Q(f_2)$.
Since the restriction of $\Phi_Q$ to the set of $Q$-preferred bases is injective, we have $f_1=f_2$.
So this process is a well-defined pre-inverse of the restriction of $\Phi_Q$.
If the input to the process is $Q$-preferred, no action is taken.
Hence this function is also a post-inverse of the restriction of $\Phi_Q$ to the set of $Q$-preferred bases.
\end{proof}

Now we present some facts about Bruhat cells.
In Section~\ref{DEFS} we associated to each $Q$-chain $\pi$ the $n \times n$ permutation matrix $s_\pi$.
It is easy to see that $s_\pi$ is $Q$-preferred when viewed as an ordered basis.

\newpage

\begin{fact} \label{union}
The set of $Q$-flags can be expressed as the union of Bruhat cells over all $Q$-chains: $\mathcal{F}\ell_Q= \bigcup \limits_{\rho} C(\rho)$.
\end{fact}

\begin{proof}
Fix any $Q$-flag $F$; find its $Q$-preferred basis $f$ as in the proof of Lemma~\ref{bijection}.
Let $\rho= (R_1, \dots, R_k)$ be the $Q$-chain for $f$ from Lemma~\ref{chain}:
The $Q$-chain $\rho$ records the pivots of $f$.
It also records the pivots of the permutation matrix $s_\rho$.
Since the pivots of both descend within each $Q$-carrel, the pivots of $f$ are the locations of the $1$s in $s_\rho$.
It can be seen that $b:=f s_\rho^{-1}$ is the matrix obtained by sorting the columns of $f$ so that all of its pivots descend from left to right.
So the matrix $b$ is upper triangular.
Then $F=\Phi_Q(f)=\Phi_Q(bs_\rho)\in C(\rho)$.
\end{proof}

\begin{fact} \label{disjoint}
Let $\pi, \rho$ be $Q$-chains.
If $\pi \neq \rho$, then the intersection of Bruhat cells $C(\pi) \cap C(\rho)$ is empty.
So for each $Q$-flag $F$, the $Q$-chain $\pi$ such that $F \in C(\pi)$ is unique.
\end{fact}

\begin{proof}
The action of $B$ preserves the pivots of an ordered basis.
So the $Q$-preferred basis for a $Q$-flag $F \in C(\pi)\cap C(\rho)$ would have the pivots listed in both $\overline \pi$ and $\overline \rho$.
But $\pi \neq \rho$.
\end{proof}

Together, these two facts show:

\begin{fact}\label{partition}
The Bruhat cells $C(\pi)$ for all $Q$-chains $\pi$ partition $\mathcal{F}\ell_Q$.  
\end{fact}

\section{Tabloid monomials, Bruhat cells, and Schubert varieties}\label{MON}

Here we present some facts concerning tabloid monomials, Bruhat cells, and Schubert varieties.
As in Section~\ref{SPAN2},  fix an $n$-partition $\lambda$ such that $Q(\lambda) \subseteq Q$.

\begin{lemma} \label{projective}
Let $g$ and $h$ be ordered bases with $\Phi_Q(g)=\Phi_Q(h)$.
There exists one $\alpha\neq 0$ such that for all tabloid monomials $\tau \in \Gamma_\lambda$, the equation $\tau(g)=\alpha \tau(h)$ holds.
\end{lemma}

\begin{proof}
Recall that column $c$ of any tabloid of shape $\lambda$ specifies a $\zeta_c$-initial minor of $h$.
Since $Q(\lambda)\subseteq Q$, the sequence of elementary column operations in the proof of Lemma~\ref{bijection} that produces the $Q$-preferred basis $f$ from $h$ is also a sequence of elementary column operations when restricted to any $\zeta_c$-initial submatrix of $h$.
Hence any $\zeta_c$-initial minor of $h$ is a nonzero multiple, say $\kappa_c(h)$, of the same $\zeta_c$-initial minor of $f$.
Then $\tau(h)=\left(\prod_{c=1}^{\lambda_1} \kappa_{c}(h)\right) \tau(f)$.
The valuation $\tau(g)$ also differs by some uniform nonzero scalar multiple from $\tau(f)$ for all tabloid monomials $\tau$.
\end{proof}

Fix a $Q$-flag $F$.
We can evaluate the sequence of all tabloid monomials of shape $\lambda$ at any ordered basis representative for $F$.
By the above lemma, the projective equivalence class of this sequence of valuations does not depend on the choice of representative.
In this way we define a map $\Omega_\lambda: \mathcal{F}\ell_Q \to \mathbb{P}(\mathbb{C}^N)$, where $N$ is the number of tabloids of shape $\lambda$.

Now fix a $Q$-chain $\pi=(P_1, \dots, P_k)$.
The next four results consider whether the tabloid monomials vanish or not at an ordered basis $h$ when $\Phi_Q(h)$ is in the Bruhat cell $C(\pi)$ or the Schubert variety $X(\pi)$.

\begin{lemma}\label{nonzero}
At any ordered basis $h$ with $\Phi_Q(h) \in C(\pi)$, the monomial $\psi_\lambda(\pi)$ of the $\lambda$-key $Y_\lambda(\pi)$ does not vanish.
\end{lemma}

\begin{proof}
By Lemma~\ref{bijection} there is a $Q$-preferred basis $f$ such that $\Phi_Q(f)=\Phi_Q(h)$.
There is a $b\in B$ such that $f=bs_\pi$.
When the columns of $f$ are sorted so that their pivots are in descending order, we produce an upper triangular matrix with $1$s on the diagonal.
Since the minor specified by any column of length $n$ is the determinant of $f$, we see that such a minor is $\pm 1$.
For any $1\leq j \leq k$, the minor of $f$ specified by a column of $Y_\lambda(\pi)$ of length $q_j$ is the determinant of the $q_j$-initial submatrix of $h$ with rows given by $P_j$.
The pivot coordinates in the first $q_j$ columns of $f=bs_\pi$ are also given by $P_j$.
Hence when the columns of this $q_j$-initial submatrix of $f$ are sorted so that these pivots are in descending order, we produce an upper triangular matrix with $1$s on the diagonal.
Multiplying these minors, we see that the value of $\psi_\lambda(\pi)$ is $\pm 1$ at $f$.
By Lemma~\ref{projective}, the value is also nonzero at $h$.
\end{proof}

\begin{lemma} \label{zero}
At any ordered basis $h$ with $\Phi_Q(h) \in X(\pi)$, any tabloid monomial $\tau \in Z_\lambda(\pi)$ vanishes.
\end{lemma}
\begin{proof}
Let $f$ be the $Q$-preferred basis such that $\Phi_Q(f)=\Phi_Q(h)$.
Then $\Phi_Q(f) \in C(\rho)$ for some $\rho \preceq \pi$.
By Lemma~\ref{order}, we have $Y_\lambda(\rho) \preceq Y_\lambda(\pi)$.
There is a $b \in B$ such that $f=bs_{\rho}$.
Since $\tau \in Z_\lambda(\pi)$, it is the monomial of a tabloid $T$ such that $T \not \preceq Y_\lambda(\pi)$.
Find a location $(r,c)\in \lambda$ such that $T(r,c)>Y_\lambda(\pi)[(r,c)]$.
The $r$ highest pivots of the first $\zeta_c$ columns of $f=bs_{\rho}$ are the coordinates $Y_\lambda(\rho)[(1,c)]$, $Y_\lambda(\rho)[(2,c)]$, $\dots$, $Y_\lambda(\rho)[(r,c)]$.
These coordinates are at or above the coordinate $Y_\lambda(\pi)[(r,c)]$ since $Y_\lambda(\rho)[(r,c)] \leq Y_\lambda(\pi)[(r,c)]$.
On the other hand, the minor in $\tau$ specified by column $c$ of $T$ is the determinant of a $\zeta_c$-initial submatrix $m$ of $f$ whose final $\zeta_c-r+1$ rows are the rows $T(r,c), T(r+1, c), \dots, T(\zeta_c, c)$ of $f$.
Since $T(r,c)>Y_\lambda(\pi)[(r,c)]$, the $r$ columns of $f=bs_\rho$ with the pivots listed above have zeros in the last $\zeta_c-r+1$ rows used for $\tau$.
This leaves at most $\zeta_c-r$ columns of $m$ which can be nonzero in their last $\zeta_c-r+1$ rows.
Hence $\det(m)=0$, and so $\tau(f)=0$.
By Lemma~\ref{projective}, we also have $\tau(h)=0$.
\end{proof}

\begin{fact}
Suppose $Q(\lambda)=Q$.  At any ordered basis $h$ with $\Phi_Q(h) \not \in X(\pi)$, there exists a tabloid monomial $\tau \in Z_\lambda(\pi)$ that does not vanish.
\end{fact}

\begin{proof}

By Fact \ref{partition}, there is a unique $Q$-chain $\rho$ such that $\Phi_Q(h) \in C(\rho)$.
Since $\Phi_Q(h) \not \in X(\pi)$, we have $\rho \not \preceq \pi$.
Since $Q(\lambda)=Q$, Lemma~\ref{order} concludes $Y_\lambda(\rho) \not \preceq Y_\lambda(\pi)$.
Hence $\psi_\lambda(\rho)\in Z_\lambda(\pi)$.
By Lemma \ref{nonzero}, the value of $\psi_\lambda(\rho)$ at $h$ is nonzero.  
\end{proof}

The preceeding three statements actually depended only on the $Q$-flag of an ordered basis, since the vanishing and nonvanishing of tabloid monomials is preserved under the scaling found in Lemma~\ref{projective}.
So these statements are useful when the tabloid monomials are used as projective coordinates for $\mathcal{F}\ell_Q$.
The last two statements show:

\begin{fact} \label{zero_set}
If $Q(\lambda)=Q$, then the Schubert variety $X(\pi)$ is the zero set in $\mathcal{F}\ell_Q$ of $Z_\lambda(\pi)$.
\end{fact}

Using the above facts, we can finally show:

\begin{fact} \label{injective}
If $Q(\lambda)=Q$, then the sequence of all tabloid monomials of shape $\lambda$ distinguishes $Q$-flags.
That is, the map $\Omega_\lambda$ is injective and faithfully parameterizes $\mathcal{F}\ell_Q$.
\end{fact}

\begin{proof}
Let $F, G$ be $Q$-flags.
Find the $Q$-preferred bases $f,g$ of $F, G$.
Let $\pi, \rho$ be the $Q$-chains such that $F \in C(\pi)$ and $G \in C(\rho)$ .
Suppose $\pi \not \preceq \rho$.
Since $Q(\lambda)=Q$, Lemma~\ref{order} concludes $Y_\lambda(\pi) \not \preceq Y_\lambda(\rho)$.
So the monomial $\psi_\lambda(\pi)$ is in  $Z_\lambda(\rho)$.
By Lemmas \ref{nonzero} and \ref{zero}, the monomial $\psi_\lambda(\pi)$ is nonzero at $f$ and zero at $g$.
If $\pi \prec \rho$, apply the argument above to $\rho \not \preceq \pi$.

Otherwise $\pi=\rho$, and so $f$ and $g$ have the same list of pivots $\overline \pi= (\pi_1, \dots, \pi_n)$.
From the proof of Lemma \ref{nonzero}, the monomial $\psi_\lambda(\pi)$ is either $1$ at both $f$ and $g$ or $-1$ at both.
Write $\pi=(P_1, \dots, P_k)$.
If $P_j=\{1, 2, \dots, q_j\}$ for all $1 \leq j \leq k$, then $\overline \pi$ is the identity permutation.
Here the identity matrix depicts the only $Q$-preferred basis, and so there is only one $Q$-flag.
So suppose there is some $1\leq j \leq k$ such that $P_j \neq \{1, 2, \dots, q_j\}$.
Then in a $Q$-preferred basis, there exists a matrix entry unconstrained by the three $Q$-preferred properties.
Now suppose $F \neq G$.
Then $f \neq g$
Find an entry $(r,c)$ where $f \neq g$.
It must lie above the pivot in column $c$, so $r<\pi_c$.
Let $1 \leq \ell \leq k$ be minimal such that $\pi_c \in P_\ell$.
One of $f$ or $g$ is nonzero at entry $(r,c)$.
By the third property of $Q$-preferred bases, the column with pivot coordinate $r$ lies in a later $Q$-carrel than does column $c$.
Hence $r \not \in P_\ell$.
Let $T$ be the tabloid obtained from $Y_\lambda(\pi)$ by replacing one of its columns $Y(P_\ell)$ with the column $Y\left(P_\ell \setminus \{\pi_c\} \cup \{r\} \right) \prec Y(P_\ell)$. 
Following a proof similar to that of Lemma \ref{nonzero}, it can be seen that the evaluation of its monomial $\tau$ at $f$ and $g$ gives (up to sign) their $(r,c)$ entries.
Therefore the two valuations of the pair of monomials $(\psi_\lambda(\pi), \tau)$ at $f$ and at $g$ are not multiples of each other. 
\end{proof}

\section{Linear independence of the Demazure monomials}\label{LIND}

The $n$-partition $\lambda$ such that $Q(\lambda)\subseteq Q$ remains fixed.
The objective of this section is to prove:

\begin{theorem}\label{LI}
Fix a nonempty $Q\subseteq [n-1]$.  Let $\lambda$ be an $n$-partition such that $Q(\lambda)\subseteq Q$ and let $\pi$ be a $Q$-chain.  The monomials of the $\pi$-Demazure tableaux of shape $\lambda$ are linearly independent in the vector space $\Gamma_\lambda(\pi)$.
\end{theorem}

A particular application of Theorem~\ref{LI} gives the linear independence of the tableau monomials for Theorem~\ref{tableaux}:

\begin{cor}\label{TabInd}
Let $\lambda$ be an $n$-partition.  The monomials of the semistandard tableaux of shape $\lambda$ form a basis of the vector space $\Gamma_\lambda$.
\end{cor}

\begin{proof}
Take $Q:=Q(\lambda)$.
Let $\pi_0$ be the maximal $Q$-chain of subsets $P_j:=\{n-q_j+1, n-q_j+2, \dots, n\}$ for $1\leq j \leq k$.
It can be seen that every tabloid of shape $\lambda$ is dominated by the $\lambda$-key $Y_\lambda(\pi_0)$ in the partial order $\preceq$.
So every tableau is $\pi_0$-Demazure.
Here, the subspace $Z_\lambda(\pi_0)$ of $\Gamma_\lambda$ is $\{0\}$.
So we have $\Gamma_\lambda(\pi_0)=\Gamma_\lambda/Z_\lambda(\pi_0)=\Gamma_\lambda$.
\end{proof}

Fix a $Q$-chain $\pi$ from now on.
Let $\xi$ be any polynomial in $\Gamma_\lambda$.
Suppose we can find an ordered basis $f$ such that its $Q$-flag $\Phi_Q(f)$ lies in the Schubert variety $X(\pi)$ and $\xi(f) \neq 0$.
By Lemma~\ref{zero}, the latter property implies that $\xi \not \in Z_\lambda(\pi)$.
Then the residue $\overline \xi$ in $\Gamma_\lambda(\pi)=\Gamma_\lambda / Z_\lambda(\pi)$ is nonzero.
Since $C(\pi) \subseteq X(\pi)$, this observation implies that the theorem follows from:

\begin{prop} \label{test}
Let $\pi$ be a $Q$-chain.
Let $T_1, \dots, T_\ell$ be $\pi$-Demazure tableaux of shape $\lambda$.
For any nonzero coefficients $a_1, \dots, a_\ell$, there is some ordered basis $f$ with $\Phi_Q(f)$ in the Bruhat cell $C(\pi)$ such that $\displaystyle \sum \limits_{i=1}^\ell a_i \tau_i(f) \neq 0$.
\end{prop}

We will prove this proposition using induction on the number of columns of $\lambda$.

Before proving this proposition, we now elaborate on the ``efficiency" claim from Section~\ref{DEFS}.
Let $N$ denote the number of tabloids of shape $\lambda$.
Consider the map from the set $GL_n$ to $\mathbb{C}^N$ given by the evaluation of the sequence of all tabloid monomials of shape $\lambda$.
The coordinatization $\Omega_\lambda: \mathcal{F}\ell_Q \to \mathbb{P}(\mathbb{C}^N)$ of Section~\ref{MON} was given by observing that the set of matrix representatives for a given flag maps to a unique projective equivalence class in $\mathbb{C}^N$.
This coordinatization is inefficient:  By the spanning Theorem~\ref{span 1}, the coordinatization of $\mathcal{F}\ell_Q$ in $\mathbb{C}^N$ up to scalar multiples is contained in a subspace $V$ that is parameterized by the coordinates corresponding to the tableau monomials.
By Proposition~\ref{test} applied to $\pi:=\pi_0$ as in the proof of Corollary~\ref{TabInd}, this subspace $V$ is the minimal subspace that contains the image of $\mathcal{F}\ell_Q$.
So we can actually coordinatize $\mathcal{F}\ell_{Q}$ with $\mathbb{P}(V) \subset \mathbb{P}(\mathbb{C}^N)$.
Let $M$ denote the number of tableau monomials of shape $\lambda$.
Then $V \cong \mathbb{C}^M$ and one may more efficiently coordinatize $\mathcal{F}\ell_Q$ in $\mathbb{P}(\mathbb{C}^M)$ by evaluating only the sequence of tableau monomials.

But this new coordinatization is inefficient for a proper Schubert variety $X(\pi)$.
By Theorem~\ref{span 2}, the coordinatization of $X(\pi)$ in $\mathbb{C}^M$ up to scalar multiples is contained in a subspace $V(\pi)$ that is parameterized by the coordinates corresponding to the $\pi$-Demazure monomials.
By Proposition~\ref{test}, this subspace $V(\pi)$ is the minimal subspace that contains the image of $X(\pi)$.
So we can actually coordinatize $X(\pi)$ with $\mathbb{P}(V(\pi)) \subset \mathbb{P}(\mathbb{C}^M)$.
Let $M(\pi)$ denote the number of $\pi$-Demazure monomials of shape $\lambda$.
Then $V(\pi) \cong \mathbb{C}^{M(\pi)}$, and one may more efficiently coordinatize $X(\pi)$ in $\mathbb{P}(\mathbb{C}^{M(\pi)})$ by evaluating only the sequence of $\pi$-Demazure monomials.

Now we assume our field has characteristic zero.
The corollary to the following proposition is used as the last step in the proof of Proposition~\ref{test}.
Here the limit in the set of ordered bases of $\mathbb{C}^n$ is found with respect to the usual metric on the $n^2$ entries of ordered bases when they are  viewed as $n \times n$ complex matrices.

\begin{prop}\label{path}
Let $\pi$ be a $Q$-chain.
Let $1 \leq i < j \leq n$ and use the reflection $\sigma_{ij}$ to define $\rho:=\sigma_{ij}\pi$.
Let $F$ be a $Q$-flag in the Bruhat cell $C(\rho)$.
If $\rho \prec \pi$, then there is a path $\beta(t)$ in the set of ordered bases of $\mathbb{C}^n$ with $\Phi_Q(\beta(t))\in C(\pi)$ for $0<t<\frac12$ such that $F = \Phi_Q\left(\lim \limits_{t \to 0} \beta(t)\right)$.
\end{prop}

\begin{proof}
Let $s_\pi$ be the $n \times n$ permutation matrix associated to $\pi$ as in Section~\ref{DEFS}.  
Construct a path $\gamma(t)$ in the space of $n \times n$ matrices by altering $s_\pi$ as follows:
Let $c_i$ and $c_j$ be the column indices such that entries  $(i,c_i)$ and $(j, c_j)$ of $s_\pi$ are $1$.
Since $\rho \neq \pi$, columns $c_i$ and $c_j$ are in different $Q$-carrels of $s_\pi$.
And since $\rho \prec \pi$, we have $c_j < c_i$.
The submatrix at rows $(i,j)$ and columns $(c_j, c_i)$ of $s_\pi$ is \small $\begin{bmatrix} 0 & 1 \\ 1 & 0 \end{bmatrix}$\normalsize.
Change these entries to \small$\begin{bmatrix} 1-t & t \\ t & 1-t \end{bmatrix}$\normalsize.

For $0<t<\frac12$, compute the $Q$-preferred basis of $\gamma(t)$ by subtracting $\frac{1-t}{t}$ times column $c_j$ from column $c_i$ and re-scaling.
Then we see that $\Phi_Q\left(\gamma(t)\right)$ is still in $C(\pi)$.
However, the limit $\lim \limits_{t\to 0} \gamma(t)$ is the permutation matrix formed from $s_\pi$ by switching columns $c_i$ and $c_j$.
Up to a reordering of columns within the affected $Q$-carrels, this is the permutation matrix $s_\rho$ for the $Q$-chain $\rho$.
So the $Q$-flag for this limit is $\Phi_Q(s_\rho)$.
Let $b\in B$ be such that $F=\Phi_Q(bs_\rho)$.
Here $F$ is also $\Phi_Q\left(b \lim \limits_{t\to 0} \gamma(t) \right)$.
Define $\beta(t):=b\gamma(t)$.
Then we have $\Phi_Q(\beta(t)) \in C(\pi)$ for $0 < t < \frac12$.
Note that $\lim \limits_{t\to 0} b \gamma(t)=b \lim \limits_{t\to 0} \gamma(t)$, since the entries in this product by $b$ are linear combinations of the original matrix entries. 
Finally we have $\Phi_Q\left( \lim \limits_{t\to 0}\beta(t) \right)=\Phi_Q\left(b \lim \limits_{t\to 0} \gamma(t) \right)=F$. 
\end{proof}

The following corollary relates the vanishing of a polynomial in $\Gamma_\lambda$ on the Bruhat cell $C(\pi)$ to its vanishing on the Schubert variety $X(\pi)$.
Its proof uses the fact that the application of a polynomial from $\mathbb{C}[x_{ij}]$ commutes with forming a limit in the $n \times n$ complex matrices.

\begin{cor} \label{closure}
Let $\pi$ be a $Q$-chain.
Let $f$ be an ordered basis with $\Phi_Q(f)\in X(\pi)$ and fix a polynomial $\xi \in \Gamma_\lambda$.
If $\xi(h)=0$ for every ordered basis $h$ with $\Phi_Q(h)\in C(\pi)$, then $\xi(f)=0$.
\end{cor}

\begin{proof}
Since $\Phi_Q(f) \in X(\pi)$, we have $\Phi_Q(f) \in C(\rho)$ for some $Q$-chain $\rho \preceq \pi$.
The conclusion is trivial if $\rho=\pi$, so suppose that $\rho \prec \pi$.
By Lemma \ref{stepdown}, there exist $1 \leq i<j \leq n$ such that $\pi_1:=\sigma_{ij} \pi$  satisfies $\rho \preceq \pi_1 \prec \pi$.  
Since there are finitely many $Q$-chains, we can iterate Lemma \ref{stepdown} until we have a sequence of reflected $Q$-chains $\rho=\pi_m \prec \pi_{m-1} \prec \dots \prec \pi_1 \prec \pi=:\pi_0$ for some $m>0$.  
Let $\ell$ run from $0$ to $m-1$ and iterate the following:
Let $h_{\ell+1}$ be any ordered basis with $\Phi_Q(h_{\ell+1}) \in C(\pi_{\ell+1})$.
Denote this $Q$-flag by $F$.
Here the proposition constructed a path $\beta(t)$ with  $\Phi_Q(\beta(t)) \in C(\pi_\ell)$ for $0 < t < \frac12$ such that $F = \Phi_Q\left(\lim \limits_{t \to 0} \beta(t)\right)$.
By induction: For every ordered basis $h_\ell$ with $\Phi_Q(h_\ell)\in C(\pi_\ell)$, we had $\xi(h_\ell)=0$.
Since $\xi \in \mathbb{C}[x_{ij}]$, we have $\xi\left(\lim \limits_{t \to 0} \beta(t)\right)= \lim \limits_{t \to 0} \xi(\beta(t))= \lim \limits_{t \to 0} 0=0$.
By Lemma~\ref{projective} applied to the ordered bases $\lim \limits_{t \to 0} \beta(t)$ and $h_{\ell+1}$ for $F$, we have $\xi(h_{\ell+1})=0$.
When beginning the $\ell=m-1$ iteration, take $h_{\ell+1}:=f$.
\end{proof}

Now we are prepared to prove Proposition~\ref{test}:

\begin{proof}[Proof of Proposition~\ref{test}]
The base case for our induction on the number of columns of $\lambda$ is when every column of $\lambda$ has length $n$, perhaps vacuously.
Here the only tableau consists of the columns $Y([n])$.
It is $\pi$-Demazure.
Its monomial is a nonnegative power of the determinant, which is nonzero.
So in this case the proposition holds.

Suppose $\lambda$ has at least one column of length less than $n$.
As for Lemma~\ref{delete}, let $\lambda'$ denote the partition obtained from $\lambda$ by deleting the rightmost column of its shape.
Note that $Q(\lambda') \subseteq Q$.
Suppose by induction that for any $Q$-chain $\rho$ and linear combination $\xi$ of $\rho$-Demazure monomials of shape $\lambda'$, there is some ordered basis $f$ with $\Phi_Q(f)\in C(\rho)$ such that $\xi(f)\neq 0$.

Write $\pi=(P_1, P_2, \dots, P_k)$.
Let $1\leq h \leq k$ be minimal such that $q_h \in Q(\lambda)$.
Examine the rightmost columns of the tableaux $T_1, \dots, T_\ell$ and identify a minimal column among these with respect to the order $\preceq$.  
Suppose $m$ of the tableaux share this minimal column, which has length $q_h$.  
Reindex the tableaux so that $T_1, \dots, T_m$ have this rightmost column.
We now form a $Q$-chain $(R_1, \dots, R_k)=:\rho$ from this minimal column and $\pi$ in such a way that $\rho$ is small enough to have $\rho \preceq \pi$ and large enough to have $Y_\lambda(\rho)\succeq S(T_i) $ for $1 \leq i \leq m$.
For $1 \leq j \leq h$, take $R_j$ to be the set of the $q_j$ smallest tableau values in this minimal column.
So this minimal column is $Y(R_h)$.
By Lemma~\ref{key} this column is also the rightmost column of $S(T_i)$ for $1 \leq i \leq m$.
Since $S(T_i)\preceq Y_\lambda(\pi)$, we have $Y(R_h)\preceq Y(P_h)$.
This implies that $Y(R_j) \preceq Y(P_j)$ for any $1 \leq j \leq h$.

For $h+1 \leq j \leq k$, form $R_{j}$ by evolving $P_{j}$ using $R_h$ as follows:
List the elements $r_1 < \dots < r_{q_h}$ of $R_h$ in increasing order.
As $t$ runs from $1$ to $q_h$, successively replace the smallest element of $P_{j}$ that is larger than or equal to $r_t$ with the element $r_t$.
Such an element exists since $P_j \supset P_h$ and $Y(P_h) \succeq Y(R_h)$.
Visualize this replacement using the column $Y(P_j)$:  by our replacement rule, replacing this value in $Y(P_j)$ by $r_t$ in the same position preserves the property that the filling increases down the column.
Define $R_{j}$ to be the set resulting from the $q_h$ iteration.
Then the column $Y(R_j)$ is produced from $Y(P_j)$ by decreasing some of its values to values from $R_h$ without reordering.
So we have $Y(R_{j})\preceq Y(P_{j})$.
For $h+1 \leq j \leq q_k$, we can see that $R_{j-1} \subset R_j$ as follows:
Let $r \in R_{j-1}$.
If $r \in R_h$, then $r \in R_j$.
On the other hand if $r \not \in R_h$, then $r \in P_{j-1} \subset P_j$.
Both $R_{j-1}$ and $R_j$ were formed by replacing elements of $P_{j-1}$ and $P_j$ respectively with elements of the same set $R_h$.
The element $r$ was not replaced when $R_{j-1}$ was formed from $P_{j-1}$, so it also was not replaced when $R_j$ was formed from $P_{j}$.
Then $\rho:= (R_1, R_2, \dots, R_k)$ is a $Q$-chain, and $\rho \preceq \pi$.

Fix $m+1 \leq i \leq \ell$.
The rightmost column of $Y_\lambda(\rho)$ is $Y(R_h)$, which was minimal among the rightmost columns of $T_1, \dots, T_\ell$.
Since $T_i$ does not share this minimal rightmost column, we can see that $T_i \not \preceq Y_\lambda(\rho)$. 
So by definition we have $\tau_{m+1}, \dots, \tau_{\ell} \in Z_\lambda(\rho)$.
Then by Lemma~\ref{zero}, at any ordered basis $f$ with $\Phi_Q(f) \in X(\rho)$ we have $\sum_{i=1} ^\ell a_i\tau_i(f) =  \sum_{i=1} ^m a_i\tau_i (f)$.

We want to show that each of $T_1, \dots, T_m$ is $\rho$-Demazure.
Fix $1 \leq i \leq m$.
We know that $S(T_i) \preceq Y_\lambda(\pi)$.
Fix a location $(b,c) \in \lambda$. 
From the construction of $\rho$, the value $Y_\lambda(\rho)[b,c]$ is $Y_\lambda(\pi)[b,c]$ or else a value from $R_h$.
Suppose $Y_\lambda(\rho)[b,c]=Y_\lambda(\pi)[b,c]$.
Since $S(T_i)\preceq Y_\lambda(\pi)$, we have $S(T_i)[b,c] \leq Y_\lambda(\rho)[b,c]$.
Now suppose $Y_\lambda(\rho)[b,c]$ is some value $r \in R_h$.
By Lemma~\ref{key}, the value $r$ appears in column $c$ of $S(T_i)$.
Let $1 \leq d \leq \zeta_c$ denote the row index such that $r=S(T_i)[d,c]$.
Since $S(T_i) \preceq Y_\lambda(\pi)$, we have $Y_\lambda(\pi)[d,c] \geq r$.
But from the construction of $\rho$, the value $Y_\lambda(\pi)[b,c]$ is the smallest value in its column larger than or equal to $r$.
Since the filling $Y_\lambda(\pi)$ increases down each column, we have $b \leq d$.
Hence $S(T_i)[b,c] \leq S(T_i)[d,c] = r = Y_\lambda(\rho)[b,c]$.
Therefore $S(T_i) \preceq Y_\lambda(\rho)$ in both cases.
So the tableaux $T_1, \dots, T_m$ are $\rho$-Demazure.
Then by Lemma \ref{chopD} the corresponding tableaux $T_1', \dots, T_m'$ of shape $\lambda'$ are $\rho$-Demazure.
By the inductive hypothesis, there is an ordered basis $f$ with $\Phi_Q(f)\in C(\rho)$ such that $\sum_{i=1} ^m a_i \tau_i'(f) \neq 0$.

The rightmost column of $Y_\lambda(\rho)$ is $Y(R_h)$.
Hence the minor specified by $Y(R_h)$ is a factor of the monomial $\psi_\lambda(\rho)$.
By Lemma~\ref{nonzero}, the value of $\psi_\lambda(\rho)$ at $f$ is nonzero.
Hence the minor specified by $Y(R_h)$ has some value $\alpha \neq 0$ at $f$.
Therefore the valuation $ \sum_{i=1} ^\ell a_i \tau_i(f)= \sum_{i=1} ^m a_i \tau_i(f)=\alpha \sum_{i=1} ^m a_i \tau_i'(f)$.
Thus we have $\sum_{i=1} ^\ell a_i \tau_i(f)\neq 0$, where $\Phi_Q(f) \in C(\rho)$.

For the sake of contradiction, suppose that for every ordered basis $h$ with $\Phi_Q(h)\in C(\pi)$ we have $ \sum_{i=1} ^\ell a_i \tau_i(h)=0$.
By design we have $\rho\preceq \pi$, and so $\Phi_Q(f) \in X(\pi)$.
Then by Corollary~\ref{closure} we also have $\sum_{i=1} ^\ell a_i \tau_i(f)= 0$, a contradiction.
\end{proof}

\section{Summation formula for Demazure polynomials}\label{CHAR}

Let $H$ be the abelian subgroup of $B$ consisting of its diagonal matrices $diag(y_1^{-1}, \dots, y_n^{-1})$.
The group $H$ acts on ordered bases from the left.
This induces an action on our polynomial subspace $\Gamma_\lambda$ of $\mathbb{C}[x_{ij}]$:
For a monomial $\tau \in \Gamma_\lambda$, an element $h \in H$, and an ordered basis $f$ of $\mathbb{C}^n$ in matrix form, one has $h.\tau(f):=\tau(h^{-1}f)$.
Since $\tau$ is a product of minors and the multiplication here by $h^{-1}$ scales the rows of $f$, we see that $\mathbb{C}\tau$ is $H$-invariant.
Given a tableau $T$, let $c_i$ be the number of values in $T$ equal to $i$.
Then the character of $H$ acting on $\mathbb{C}\tau$ is $y^{T}:=\prod_{i=1}^n y_i^{c_i}$.

Now fix a $Q$-chain $\pi$.
The subspace $Z_\lambda(\pi)$ is $H$-invariant.
The character of the induced representation on $\Gamma_\lambda(\pi):=\Gamma_\lambda/Z_\lambda(\pi)$ follows from Theorem~\ref{main}:

\begin{cor}\label{character}
The character of $H$ on $\Gamma_\lambda(\pi)$ is $\sum y^{T}$, where the sum runs over all $\pi$-Demazure tableaux of shape $\lambda$.
\end{cor}

This polynomial is the \emph{Demazure polynomial} of \cite{PW}.
This terminology will be justified in the appendix, where we note that $\Gamma_\lambda(\pi)^*$ is a Demazure module for $B$.
Given that the scanning tableau of a tableau $T$ is the right key of $T$,  this polynomial is also the ``key polynomial" of Lascoux and Sch\"utzenberger \cite[Theorem 1]{RS1}.

\section{Appendix:  Contemporary terminology}\label{APP}

Here we provide a dictionary for relating the objects of this paper to the contemporary algebraic geometry literature.
We also identify the character from Section~\ref{CHAR} using the representation theory of $GL_n$.
Continue to use the definitions from Section~\ref{CHAR}.
Here we require an algebraically closed field of characteristic zero; we use $\mathbb{C}$.

Our subgroup $H$ of diagonal matrices in $GL_n$ is called the \emph{torus}, and our subgroup $B$ of upper triangular matrices in $GL_n$ is called the \emph{Borel} subgroup.
Fix a nonempty $Q \subseteq \{1, 2, \dots, n-1\}$ and set $k:=|Q|$.
Let $E$ be the $Q$-flag of subspaces $V_j=$span$(\{e_i| i \leq q_j\})$ for $1 \leq j \leq k$.
The action of $GL_n$ on ordered bases of $\mathbb{C}^n$ induces an action on the set $\mathcal{F}\ell_Q$ of $Q$-flags.
Let $P$ be the ``parabolic" subgroup of $GL_n$ that stabilizes $E$.
Note that $B$ stabilizes $E$, so $B \subseteq P$.
Let $(v_1, \dots, v_n)$ be an ordered basis and let $f:=[v_1, \dots, v_n] \in GL_n$ be the corresponding invertible matrix.
Let $F$ be any $Q$-flag.
If $\Phi_Q(f)=F$, then $f.E=F$.
So Lemma~\ref{bijection} implies that the action of $GL_n$ on $\mathcal{F}\ell_Q$ is transitive.
From the definition of $P$, we see that $\mathcal{F}\ell_Q$ is isomorphic to the coset space $GL_n/P$ as a $GL_n$-set.
The three operations from the proof of Lemma~\ref{bijection} generate the right action of $P$ on $GL_n$ when $GL_n$ is considered as the set of all ordered bases for $\mathbb{C}^n$.
That lemma found a preferred representative in $GL_n$ for each coset in $GL_n/P$.
So the map $\Phi_Q$ can be used to describe an isomorphism from $GL_n/P$ to $\mathcal{F}\ell_Q$.

Let $\mathfrak{g}, \mathfrak{h}, \mathfrak{b}$, and $\mathfrak{p}$ denote the Lie algebras of $GL_n, H, B$, and $P$ respectively.
The Lie algebra $\mathfrak{g}$ is reductive.
Let $\phi_1, \dots, \phi_n$ denote the basis of $\mathfrak{h}^*$ such that $\phi_i(h)$ is the entry of $h$ in position $(i,i)$ for any $h \in \mathfrak{h}$.
Equip $\mathfrak{h}^*$ with the inner product for which $\phi_1, \dots, \phi_n$ is an orthonormal basis.
For each $1 \leq j \leq n-1$ set $\alpha_j:=\phi_j-\phi_{j+1}$ and $\omega_j:= \sum_{i=1}^{j} \phi_i$.
For the semisimple part of $\mathfrak{g}$, the $\alpha_1, \dots, \alpha_{n-1}$ depict the positive simple roots and the $\omega_1, \dots, \omega_{n-1}$ depict the fundamental weights.
Set $\omega_n:= \sum_{i=1}^n \phi_i$.
This weight is orthogonal to $\alpha_1, \dots, \alpha_{n-1}$; it corresponds to the center of $\mathfrak{g}$.
Set $J:=[n-1] \setminus Q$.
It can be seen that $\mathfrak{p}$ is the direct sum of $\mathfrak{b}$ and the root subspaces for the roots in the negative span of $\{ \alpha_j | j \in J \}$.
For each weight $\mu \in \mathfrak{h}^*$, there is a corresponding character $exp(\mu)$ of the torus $H$.
For $1 \leq i \leq n$ set $x_i:=exp(\phi_i)$.
Let $\lambda$ be an $n$-partition.
For $1\leq i \leq n$, set $a_i:=\lambda_i - \lambda_{i+1}$; this is the number of columns of length $i$ in the shape of $\lambda$.
Then we have $Q(\lambda)=\{1 \leq i \leq n-1 | a_i \neq 0\}$.
Use $\lambda$ to also denote the weight $\sum_{i=1}^n \lambda_i \phi_i = \sum_{i=1}^n a_i \omega_i$.
Let $V_\lambda$ denote an irreducible representation of $GL_n$ with highest weight $\lambda$.

The contragredient representation of $H$ on $\Gamma_\lambda$ defined in Section~\ref{CHAR} extends to a representation of $GL_n$:
Since a monomial $\tau$ of a tableau $T$ is a product of minors, it can be seen that $g.\tau$ is again a polynomial in $\Gamma_\lambda$.
Here $\tau$ is a weight vector of $\Gamma_\lambda$ of weight $\sum_{i=1}^n -c_i \phi_i$.

Now fix an $n$-partition $\lambda$ such that $Q(\lambda)\subseteq Q$.
Let $\epsilon$ be the minimal $Q$-chain of subsets $E_j:= \{1 , \dots, q_j\}$ for $1\leq j \leq k$.
Notate the $\lambda$-key monomial $\psi_\lambda(\epsilon)$ with $\psi$.
It can be seen that for all ordered bases $f$ and any $p \in P$, we have $\psi(fp )=\theta_\lambda(p) \psi(f)$ for a certain scalar $\theta_\lambda(p)$.
Since the function $\theta_\lambda$ on $P$ is multiplicative, it defines a character of $P$ that is realized in $GL(\mathbb{C}\psi)$.
Define an equivalence relation $\sim$ on $GL_n \times \mathbb{C}\psi$ by setting $(g, z) \sim (gp,\theta_\lambda(p) z \psi)$ for any $g \in GL_n$ and $p \in P$ and $z \in \mathbb{C}$.
Define a line bundle $\mathcal{L}_\lambda$ on $GL_n / P$ to be $(GL_n \times \mathbb{C}\psi)/ \sim$.
There is a contragredient representation of $GL_n$ on its space of global sections $\Gamma(GL_n/P, \mathcal{L}_\lambda)$:
For $\xi \in \Gamma(GL_n/P, \mathcal{L}_\lambda),$ a matrix $g \in GL_n,$ and coset $f \in GL_n/P$, we define $g.\xi(f):=\xi(g^{-1}f)$.
The Borel-Weil theorem says  \cite[Section 4]{Borel} that this representation is irreducible with lowest weight $-\lambda$.

For the monomial $\tau$ of any tabloid $T$, we more generally have $\tau (fp) =  \theta_\lambda(p) \tau(f)$ for any $f \in GL_n$ and $p \in P$.
This is because the right multiplication of ordered bases by $p$ is generated by the column operations of Lemma~\ref{bijection}, while the filling of $T$ specifies the rows used to form minors for $\tau$.
So $\tau$ can be used to define a section of $\mathcal{L}_\lambda$ that sends the coset $fP$ of $GL_n/P$ to the equivalence class $[f, \tau(f)\psi]$ of $(GL_n \times \mathbb{C}\psi)/ \sim$.
Hence $\Gamma_\lambda$ can be viewed as a  submodule of the global section space $\Gamma(GL_n/P, \mathcal{L}_\lambda)$ of this bundle.
Since $\Gamma(GL_n/P, \mathcal{L}_\lambda)$ is irreducible, this entire space is realized by $\Gamma_\lambda$.
It can bee seen that the section defined by $\psi$ is a lowest weight vector of $\Gamma(GL_n/P,\mathcal{L}_\lambda)$ for the lowest weight $-\lambda$.
Here Theorem~\ref{tableaux} says that the (semistandard) tableau monomials describe a basis for $\Gamma(GL_n/P, \mathcal{L}_\lambda)$.
Moreover, this basis is a weight basis.
Since the lowest weight of $\Gamma(GL_n/P, \mathcal{L}_\lambda)$ is $-\lambda$, the highest weight of $\Gamma(GL_n/P, \mathcal{L}_\lambda)^*$ is $\lambda$.
Hence $\Gamma_\lambda^* \cong V_\lambda$.
Since we have allowed $\lambda$ to have columns of length $n$, we can have positive powers of the determinant in our characters.
Hence each of the irreducible polynomial representations of $GL_n$ can be realized with some $\Gamma_\lambda^*$ (for all $Q$).
If suitable notation were introduced, our treatment could also handle negative powers of the determinant.
Then each of the irreducible rational representations of $GL_n$ could be realized with some $\Gamma_\lambda^*$.

The Weyl group $W$ of the semisimple part of $\mathfrak{g}$ is generated by the simple reflections $s_1, \dots, s_{n-1}$ corresponding to the simple roots.
Using the depiction of the simple roots in $\mathfrak{h}^*$ above, we can depict the action of a simple reflection on the $\phi$ basis as follows:
For $1 \leq i \leq n-1$ we have $s_i.\phi_i=\phi_{i+1}$ and $s_i.\phi_{i+1}=\phi_i$, with $s_i.\phi_j=\phi_j$ for all other $j$.
By considering only the subscripts here, we can model the action of $W$ with the group $S_n$ of permutations of $[n]$.
Corresponding to the simple reflection $s_i$, the transposition $(i,i+1)$ swaps the values $i$ and $i+1$ in an  $n$-sequence of values from $[n]$.
Given a permutation $\pi \in S_n$, write the result of $\pi.(1,2, \dots, n)$ in one-row form as $(\pi_1, \pi_2, \dots, \pi_n)$.
Then $\pi$ models the element $w \in W$ such that $w.\phi_i = \phi_{\pi_i}$ for $1 \leq i \leq n$.

The length of an element $w \in W$ is the smallest number of simple reflections needed to express $w$.
Let $W_J$ denote the subgroup of $W$ generated by the reflections $s_j$ for $j \in J$.
Since $Q(\lambda)\subseteq Q$, it can be seen that the group $W_J$ stabilizes $\lambda$.
Each coset of $W/W_J$ has a unique minimal length representative.
Let $W^J \subseteq W$ denote the set of such representatives.
It can be seen that each $Q$-permutation (Section~\ref{SETUP}) models some $w \in W^J$ and any $w\in W^J$ is correspondingly modeled by some $Q$-permutation.
So the map sending a $Q$-chain $\pi$ to the $Q$-permutation $\overline \pi$ can be viewed as a bijection from the set of $Q$-chains to $W^J$.
Under this bijection, our partial order $\preceq$ on $Q$-chains agrees \cite[Theorem 2.6.3]{BB} with the Bruhat order on $W$ restricted to $W^J$.
The Weyl group can also be depicted in $GL_n$ relative to $H$ as the group of $n \times n$ permutation matrices.
Here the $Q$-chain $\pi$ is represented by the matrix $s_\pi$ from Section~\ref{DEFS}.

Given $w \in W$, let $v_{w\lambda}$ be a weight vector in $V_\lambda$ of weight $w\lambda$.
Let $D_\lambda(w)$ denote the Demazure $B$-module $\mathbb{C}[B].v_{w\lambda}$.
Since $W_J$ stabilizes $\lambda$, the module $D_\lambda(w)$ only depends on the coset of $w$ in $W/W_J$.
So we can name this Demazure module $D_\lambda(\pi)$, where $\pi$ is the $Q$-chain corresponding to the representative of this coset in $W^J$. 

Using $Q$-preferred bases, it is can be seen that the flags $\varphi(\pi)$ for $Q$-chains $\pi$ are exactly the $H$-invariant $Q$-flags.
The Bruhat cells are the $B$-orbits of $GL_n/P$.
Corollary~\ref{closure} can be strengthened as follows:  
Given a $Q$-chain $\pi$, the Schubert variety $X(\pi)$ is the closure of the Bruhat cell $C(\pi)$ in the Zariski topology on $GL_n/P$.
This is proved over any algebraically closed field in e.g. \cite[Section 10.6]{Procesi}, but since $GL_n$ is ``split" that proof works here over any field \cite{Humph}.
If one accepts this substitute for Corollary~\ref{closure}, then every result in this paper other than Proposition~\ref{path} is valid over any field.

Now fix a $Q$-chain $\pi$.
Let $\mathcal{L}_\lambda (\pi) $ denote the restriction of $\mathcal{L}_\lambda$ to the Schubert variety $X(\pi)$.
The global section space $\Gamma(X(\pi), \mathcal{L}_\lambda(\pi))$ of this bundle is not a $GL_n$-module since $X(\pi)$ is not $GL_n$-invariant in $GL_n/P$.
But $X(\pi)$ is $B$-invariant, and so the restriction of the $GL_n$ representation on $\Gamma(GL_n/P, \mathcal{L}_\lambda)$ to the subgroup $B$ induces a representation of $B$ on $\Gamma(X(\pi), \mathcal{L}_\lambda(\pi))$.
It is known \cite{Demazure} that its dual is isomorphic to the Demazure module defined above: $\Gamma(X(\pi), \mathcal{L}_\lambda(\pi)) ^* \cong D_\lambda(\pi)$.
The section defined by the monomial $\psi_\lambda(\epsilon)$ is again a lowest weight vector of $\Gamma(X(\pi), \mathcal{L}_\lambda(\pi))$ for the lowest weight $-\lambda$.
The section defined by the monomial $\psi_\lambda(\pi)$ is a highest weight vector of  $\Gamma(X(\pi), \mathcal{L}_\lambda(\pi))$ for the highest weight $-w\lambda$, where $w \in W$ is modeled by the $Q$-permutation $\overline \pi$.
Analagously, the vector space $\Gamma_\lambda(\pi):=\Gamma_\lambda / Z_\lambda(\pi)$ is not a $GL_n$-module since $Z_\lambda(\pi)$ is not $GL_n$-invariant in $\Gamma_\lambda$.
But one can see that the action of $B$ on a tabloid monomial produces a combination of monomials for tabloids with larger values.
Then $Z_\lambda(\pi)$ is $B$-invariant, and so the restriction of the $GL_n$ representation on $\Gamma_\lambda$ to the subgroup $B$ induces a representation of $B$ on $\Gamma_\lambda(\pi)$.
Fact~\ref{zero_set} and the isomorphism $\Gamma_\lambda \cong \Gamma(GL_n/P, \mathcal{L}_\lambda)$ above imply that these $B$-modules $\Gamma_\lambda(\pi)$ and $\Gamma(X(\pi), \mathcal{L}_\lambda(\pi))$ are isomorphic.
Here Theorem~\ref{main} says that the $\pi$-Demazure monomials describe a basis for $\Gamma(X(\pi), \mathcal{L}_\lambda(\pi))$.
Moreover, this basis is a weight basis.
By this isomorphism, we have $\Gamma_\lambda(\pi)^* \cong D_\lambda(\pi)$.
Corollary~\ref{character} gives the character of $D_\lambda(\pi)$ as the Demazure polynomial $\sum x^T$, where the sum runs over all $\pi$-Demazure tableaux of shape $\lambda$.
This implies that the dimension of $D_\lambda(\pi)$ is the number of $\pi$-Demazure tableaux of shape $\lambda$.
See the appendix of \cite{PW} for more information concerning the concrete description of the coordinatized Demazure modules of $B\subset GL_n$. 

\section*{Acknowledgements} 
Special thanks to my advisor Bob Proctor for his guidance concerning the exposition of this article.
Thanks also to Shrawan Kumar for sharing his expertise on algebraic groups, to Matt Willis for his scanning tableaux and feedback on the first draft of this paper, and to Joe Seaborn for helpful discussions.
The author benefitted from the support of the Tom Brylawski Memorial Fellowship from the University of North Carolina Graduate School.  

\bibliography{Bibliography}

\end{document}